\documentclass{amsart}

\usepackage{amsmath}
\usepackage{amssymb}
\usepackage{amsthm}
\usepackage{mathrsfs}
\usepackage{color}
\usepackage{graphicx}
\usepackage{float}
\usepackage[font=footnotesize]{subcaption}
\usepackage{caption}
\usepackage{hyperref}
\usepackage{enumerate}

\theoremstyle{definition}

\theoremstyle{remark}

\numberwithin{equation}{section}

\newtheorem{theorem}{Theorem}[section]

\newtheorem{definition}[theorem]{Definition}

\newtheorem{lemma}[theorem]{Lemma}

\newtheorem{remark}[theorem]{Remark}
\newtheorem{assumption}[theorem]{Assumption}

\newcommand{\tstep}{{\Delta t}}

\newcommand{\bc}{\mathbf{c}}
\newcommand{\bv}{\mathbf{v}}

\newcommand{\scl}{\varrho}
\newcommand{\uex}{u^{\text{ex}}}

\newcommand{\df}{\mathrm{d}}

\renewcommand{\Re}{\operatorname{Re}}
\renewcommand{\O}{\mathcal{O}}
\renewcommand{\o}{o}

\newcommand{\betaA}{\beta_{\text{A}}}
\newcommand{\betaB}{\beta_{\text{B}}}
\graphicspath{ {figs/} } 

\usepackage[textwidth = 4cm]{todonotes}

\begin{document}

\title[Time-fractional Westervelt equation]{Numerical analysis of a time-stepping method for the Westervelt equation with time-fractional damping}

\author{Katherine Baker}
\address{Department of  Mathematics, Heriot-Watt University}
\email{k.baker@hw.ac.uk}

\author{Lehel Banjai}
\address{Department of Mathematics, Heriot-Watt University}
\email{l.banjai@hw.ac.uk}

\author{Mariya Ptashnyk}
\address{Department of Mathematics, Heriot-Watt University}
\email{m.ptashnyk@hw.ac.uk}




\begin{abstract}
We develop a numerical method for the Westervelt equation, an important equation in nonlinear acoustics, in the form where the attenuation is represented by a class of non-local in time operators. A semi-discretisation in time based on the trapezoidal rule and A-stable convolution quadrature is stated and analysed. Existence and regularity analysis of the continuous equations informs the stability and error analysis of the semi-discrete system. The error analysis includes the consideration of the singularity at $t = 0$ which is addressed by the use of a correction in the numerical scheme. Extensive numerical experiments confirm the theory.
\end{abstract}

\maketitle

\section{Introduction}

We consider the attenuated Westervelt equation modelling wave propagation through lossy media in cases where the wave propagation is poorly approximated by linear wave models. A typical application is in medical ultrasound, where the attenuation depends on a fractional power of the frequency with the fractional exponent determined by the type of tissue; see \cite[Chapter 4]{szabo:book}. This leads to models of the form
\[
\partial_t^2 u -\Delta u + a L u = k\partial_t^2 (u^2),
\]
 where $a,k$ are positive constants and the attenuation is represented by a nonlocal differential operator $L$. In this paper we consider $L v(t) =   -\int_0^t\beta(t-s) \partial_t \Delta v(s) ds$ with $\beta$ chosen as either
\[
\betaA(t) := \frac1{\Gamma(\mu)}t^{\mu-1} e^{-r t}, \qquad \mu \in (0,1), \; r \geq 0
\]
or 
\[
\betaB(t) := -\dot e_\mu(t), \qquad e_\mu(t):= E_{\mu,1}(-t^\mu),
\]
where, see \cite{Mainardi_ML},  $E_{\mu,\gamma}$ is the Mittag-Leffler function
\begin{equation}\label{eq:ML}
E_{\mu,\gamma}(z) := \sum_{k = 0}^\infty \frac{z^k}{\Gamma(\mu k + \gamma)}.
\end{equation}

Note that, for  $\beta = \betaA$ and $r = 0$, $L = -\partial_t^{1-\mu} \Delta$, where $\partial_t^{1-\mu}$ is the Caputo fractional derivative of order $1-\mu$. The value of $\mu$ depends on the tissue \cite[Chapter 4.3]{duck:book} and is used to model the frequency dependence of attenuation \cite[Chapter 3]{Baker2022}. See also the recent \cite{Barbara}, which includes other choices of nonlocal attenuation operators $L$. The case $\beta = \betaB$ is of interest in modelling viscoelastic materials \cite{Larsson, EndreZener}. In \cite{xu} a similar system is investigated under trapezoidal disretisation, where the operator $L$ does not contain a time-derivative.


In this work we develop and analyse a numerical method for the time-discretisation of the attenuated Westervelt equation stated above. The time-discretisation of the non-local operator is done by convolution quadrature \cite{Lubich1986,Lubich1988I} whose ability to translate a positivity property of the continuous operator to the discrete case allows a full stability and convergence analysis. It further allows for fast and memory efficient implementation \cite{banjai2019} that is not addressed further in this paper. The full time discretisation is a variation of the discretisation used in \cite{Baker_Banjai} for a related linear model and is based on the trapezoidal (Newmark with $\gamma = 1/2$, $\beta = 1/4$ \cite{hughes_fem}) scheme.

There are several results on the well-posedness and regularity for quasilinear wave equations  and  for the   Westervelt equations, see e.g.~\cite{Doerfler, Barbara_2009, Nikolic_2016}.  
In \cite{Larsson, Saedpanah},  semigroup techniques and the Galerkin method are used to prove  well-posedness results for linear integro-differential equations modelling dynamics of fractional order viscoelasticity. For equations with fractional integrals the semigroup techniques can be applied in the same way as in the case of equations of linear viscoelasticity \cite{Fabiano, Fabrizio}. However similar approach cannot be used to prove existence results for equations with nonlocal {\em differential} operators, which include fractional time derivative as a special case, considered in this work. 
The Galerkin method, together with the fixed point argument, is applied in  \cite{Barbara} for the well-posedness analysis of  fractional Westervelt equations. 
Galerkin approximation, together with the energy estimates, is also used in \cite{EndreZener} to  prove existence of weak solutions to the fractional Zener wave equations for heterogeneous viscoelastic materials. In the proof of existence and uniqueness results for the nonlocally attenuated Westervelt equation considered here we follow similar ideas as in \cite{Barbara}, and hence include only the main steps of the proof. 

The literature on the numerical methods for the case of local strong damping, i.e.,  $L = -\Delta \partial_t$  includes the semi-discretisation by continuous \cite{NikWohl:cG} and discontinuous  \cite{ANTONIETTI2020109484} Galerkin finite element methods.  Let us also mention the recent approach via semi-groups to the analysis of the spatial discretisation of a large class of quasilinear wave equations \cite{HochM:spatial}. Analysis of a fully discrete scheme for nonlinear elastic waves with the finite element method in space and rational approximation in time is presented in \cite{Babis:1993}.  

In the linear case ($k = 0$), the literature also includes the numerical analysis of full discretisations of non-local attenuations. Namely a weaker form of non-local attenuation than we are interested in ($Lv(t) = -\int_0^t \beta(t-s)\Delta v(s)ds$) is investigated in \cite{Larsson} where a continuous Galerkin semi-discretisation is analysed. In the already mentioned work \cite{Baker_Banjai} a fully discrete scheme is investigated with again weaker attenuation $L = \partial_t^{\gamma}$, $\gamma \in (0,1)$. The fully discrete scheme in \cite{Baker_Banjai}  consists of continuous Galerkin method in space and leapfrog combined with convolution quadrature in time. This was extended to the strongly damped nonlocal case (still with $k = 0$) in the thesis \cite[Chapter 6]{Baker2022} with the explicit leapfrog scheme replaced by the implicit trapezoidal time-stepping. This numerical approach we now extend to the nonlinear case to give what we believe to be the first analysis of a time-discretisation of the nonlocally attenuated Westervelt equation. The analysis also includes realistic assumptions on the regularity of the solution including the possible lack of smoothness at $t = 0$. 

The paper consists of six sections, the first being this introduction. In the next section we give the formulation of the mathematical model and prove an important property of $L$. Section 3  briefly gives the well-posedness of the nonlinear system with some of the technical details of the proof relegated to the appendix. In Section 4 we state the numerical scheme and show its stability. This leads to the proof of convergence estimates in Section 5. Finally, the results are illustrated by numerical experiments in one and two spatial dimensions in Section 6.

\section{Formulation of mathematical model}\label{formul_prob}
We start with the formulation of the mathematical model. In the damping term we shall consider a class of  convolution kernels which includes fractional time derivative as a special case. 

Let $\Omega \subset \mathbb R^d$, with $d\leq 3$,  be   a $C^{1,1}$  domain, or for $d=2$ a polygon with edge opening angles $\omega< \pi$  and for $d=3$ a  polyhedron with  $\omega\leq \pi/2$.
We consider
\begin{equation}\label{eq:1}
\partial_t^2 u -\Delta u -a \beta \ast \Delta \partial_t u = k\partial_t^2 (u^2), \end{equation}
where $a,k>0$ are constants,
\[
f \ast g (t) := \int_0^t f(t-\tau) g(\tau) d\tau
\]
denotes the one sided convolution, and $\beta$ is chosen as either
\begin{equation}
  \label{eq:betaA}
\betaA(t) := \frac1{\Gamma(\mu)}t^{\mu-1} e^{-r t},  \qquad r \geq 0, \;  \mu \in (0,1),
\end{equation}
or 
\begin{equation}
  \label{eq:betaB}
\betaB(t) := -\dot e_\mu(t), \qquad e_\mu(t) = E_{\mu,1}(-t^\mu), \qquad \mu \in (0,1),
\end{equation}
with $E_{\mu,\gamma}$  the Mittag-Leffler function \eqref{eq:ML}. In both cases
\[
\beta(t) \sim \frac1{\Gamma(\mu)} t^{\mu-1} \qquad \text{ as } \; \; t \to 0^+.
\]
When a result holds for both kernels, we will use $\beta$ to denote either of the kernels.  Note that for $\beta_A$ and  $r  = 0$, 
\[
\betaA \ast f = I_t^\mu f \quad \text{ and }  \quad \betaA \ast \partial_t f = \partial_t^{1-\mu} f,
\]
where $I_t^\mu$ denotes the Riemann-Liouville fractional integral of order $\mu \in (0,1)$ and $\partial_t^{1-\mu}$ the Caputo derivative of order $1-\mu \in (0,1)$ \cite{podlubny}.

We will need two properties of $\beta$. First of all,  denoting by $\hat \beta := \mathscr{L}\{\beta\}$   the Laplace transform of $\beta$ we have that
\begin{equation}
  \label{eq:hatbeta}  
\hat{\beta}_{\text{A}}(z) = (z+r)^{-\mu}, \qquad \hat{\beta}_{\text{B}}(z) = \frac1{z^\mu+1}.
\end{equation}
The expression for $\hat{\beta}_{\text{B}}(z)$ is obtained from the fact that the Laplace transform of $e_\mu$ is given by $\frac{z^{\mu-1}}{z^\mu+1}$, see \cite{Mainardi_ML}, and the calculation
\[
\mathscr{L}\{-\dot e_\mu\}(z) = -\left(\frac{z^{\mu}}{z^\mu+1}-1\right) = \frac1{z^\mu+1},
\]
where we used that $e_\mu(0) = 1$. 

Thus, a property we will require later, follows:
\begin{equation}
  \label{eq:hatbeta_lower}
\Re \frac1{\hat{\beta}_{\text{A}}(z)}  \geq (\sigma +r)^\mu,\qquad
\Re \frac1{\hat{\beta}_{\text{B}}(z)}  \geq 1 \qquad \forall \Re z \geq \sigma > 0.
\end{equation}
The second property we need is stated as a lemma.

\begin{lemma}\label{lem:time_lower}
For  any $v \in L^2(0,T)$ we have 
\[
\int_0^t [\beta\ast v] (s)\, v(s) ds
\geq \frac12 \min_{s \in [0,t]} \left(\gamma(t-s)+\gamma(s)\right)
\int_0^t \left|\beta\ast v(s)\right|^2 ds, \qquad t \in (0,T),
\]
where for  $\beta = \betaA$ 
\[
\gamma(t) = \frac1{\Gamma(1-\mu)}e^{-r t}t^{-\mu}+ \frac r{\Gamma(1-\mu)}\int_0^t \tau^{-\mu}e^{-r \tau}d\tau 
\]
and for  $\beta = \betaB$
\[
\gamma(t) = \frac1{\Gamma(1-\mu)} t^{-\mu}+1.
\]
\end{lemma}
\begin{proof}
Note that $\gamma$ is chosen so that $\frac1z = \hat \gamma(z) \hat \beta(z)$ and hence
\begin{equation}
  \label{eq:vgamma}
\int_0^t v(s)ds =   \int_0^t \gamma(t-\tau) \int_0^\tau \beta(\tau-\eta) v(\eta) d\eta d\tau,
\end{equation}
for any sufficiently smooth $v$.

Denoting $w = \beta \ast  v$, we have, by differentiating \eqref{eq:vgamma}, multiplying by $w$ and integrating, that
\[
\int_0^t w(s) v(s) ds = 
\int_0^t w(s)\frac{d}{ds} \int_0^s \gamma(s-\tau)w(\tau)d\tau ds.
\]
We complete the proof by noticing  that $\gamma$ satisfies the conditions of the kernel $k$ in \cite[Lemma~3.1]{EndreZener} with the lemma thus implying
\[
  \begin{split}    
\int_0^t \frac{\df}{\df s}\left[\int_0^s \gamma(s-\tau) w(\tau) \df\tau \right]w(s) \df s
&\geq \frac12 \int_0^t \left[\gamma(t-s)+\gamma(s)\right] |w(s)|^2 \df s\\
&\geq \frac12 \min_{s \in [0,t]} \left[\gamma(t-s)+\gamma(s)\right] \int_0^t |w(s)|^2 \df s.
  \end{split}
\]
Finally note that since $\beta \in L^1(0,T)$, Young's inequality for convolutions implies that both sides of the above inequality are well-defined for $v \in L^2(0,T)$ thus completing the proof.
\end{proof}

\begin{remark}
  An application  of Plancherel's formula as in \cite[Lemma 2.2]{BaLuSa:2015},  shows that from \eqref{eq:hatbeta_lower} it follows that
\[
\int_0^\infty e^{-2\sigma s} [\beta \ast v](s)  v(s) ds \geq C(\sigma) \int_0^\infty e^{-2\sigma s} |\beta \ast v(s)|^2 ds
\]
for $\sigma > 0$ and $C(\sigma) = (\sigma+r)^\mu$ for $\beta = \betaA$ and $C(\sigma) = 1$ for $\beta = \betaB$. While it would be possible to develop the theory in the next section based on this inequality, it is easier to use Lemma~\ref{lem:time_lower}.
\end{remark}

We rewrite equation \eqref{eq:1} as
\begin{equation}\label{main_integro}
\begin{aligned}
  (1-2ku) \partial_t^2 u -\Delta u - a \beta \ast \partial_t \Delta u  & = 2k (\partial_t u)^2  && \text{in } (0,T) \times \Omega, \\
 u& =u_D  && \text{on } (0,T)\times \partial \Omega, \\ 
 u(0)=u_0, \quad \partial_t u(0)& = v_0  && \text{in } \Omega
 \end{aligned}
\end{equation}
and make the following assumptions on the smoothness of the data:
\begin{equation} \label{init_cond_assum}
\begin{aligned} 
 & u_0-u_D(0) \in \dot{H}^3(\Omega) \cap H^1_0(\Omega), \quad v_0 \in H^2(\Omega)\cap H^1_0(\Omega), \quad \\
 & u_D \in H^2(0,T; H^2(\Omega))\cap L^\infty(0,T; H^3(\Omega)),
 \end{aligned}
\end{equation}
 where $\dot H^3 = \mathcal D ((-\Delta_D)^{3/2} u)$ with the norm $\|w\|_{\dot H^3} = \|(-\Delta_D)^{3/2} w \|_{L^2}$ and $\Delta_D$ is the Laplace operator in $L^2(\Omega)$ with the zero Dirichlet boundary conditions, i.e.~with domain $\mathcal D(\Delta_D) = H^2(\Omega) \cap H^1_0(\Omega)$.

\begin{definition}
 A weak solution of \eqref{main_integro} is function $u \in u_D +  H^1(0,T; H^1_0(\Omega))$, with $u \in L^\infty((0,T)\times \Omega)$ and  $\partial_t^2 u \in  L^2((0,T)\times \Omega)$, satisfying
 \begin{equation}\label{weak_sol_def}
 \begin{aligned}
  \int_0^T \big[\langle (1- 2ku)  \partial_t^2 u,  \phi \rangle  +  \langle \nabla u + a \beta \ast \partial_t \nabla u, \nabla \phi \rangle\big] dt
  = \int_0^T \langle 2k (\partial_t u)^2, \phi \rangle dt,
  \end{aligned}
 \end{equation}
for $\phi \in L^2(0, T; H^1_0(\Omega))$, and initial conditions are satisfied in the $L^2$-sense.
\end{definition}

For the simplification of the presentation,  we shall consider $u_D=0$, however all results hold for non-zero Dirichlet boundary conditions by considering $\hat u = u - u_D$, resulting in
\begin{equation}\label{eq_with_u_D}
\begin{aligned}
  ((1- 2k u_D)-2k\hat u) \partial_t^2 \hat u -\Delta \hat u - a \beta \ast \partial_t \Delta \hat u   = 2k (\partial_t \hat u)^2 +  f(t,x)\\
  + 4k \partial_t \hat u \partial_t u_D + 2k\hat u \partial_t^2 u_D,   
  \end{aligned}
  \end{equation}
where $f(t,x)= 2k (\partial_t u_D)^2 + (2k u_D-1) \partial_t^2 u_D + \Delta u_D + a \beta \ast \partial_t \Delta u_D$. For $u_D$ independent of $t$, the difference between \eqref{main_integro} and \eqref{eq_with_u_D} is in the presence of function $f(t,x)$, which is regular for regular $u_D$ and the analysis below holds for all sufficiently regular $u_D$ with $\|u_D\|_{L^\infty(\Omega)} < 1/(2k)$.  In case $u_D$ depends on $t$ we obtain additional linear terms, which can be treated in the same way as in the case $u_D=0$.  


\section{Existence and uniqueness results}\label{sec:nonlinear} 

We shall apply the Banach fixed-point theorem and the Galerkin method to show existence and uniqueness  of solutions of \eqref{main_integro}. Similar approach was considered  in~\cite{Barbara}, however for completeness we present here the short outline of the main ideas. Also we have a more general convolution kernel, compared to the one considered in~\cite{Barbara}.

 For $\Omega \in C^{1,1}$ the elliptic regularity theory, see e.g.~\cite[Theorem 9.15, Lemma 9.17]{GT},  ensures
\begin{equation} \label{H2_Delta}
\| w \|_{W^{2, p}(\Omega)} \leq C_\Omega \|\Delta w\|_{L^p(\Omega)},
\end{equation}
for $w \in H^1_0(\Omega)$ with $\Delta w \in L^p(\Omega)$ and $p \in (1, \infty)$, and some positive constant $C_\Omega$, depending on the domain $\Omega$.
For polygons estimate \eqref{H2_Delta} holds  for  $1<p < 2 \omega/(2\omega - \pi)$, see e.g.~\cite[Theorem~4.3.2.4, Remark~4.3.2.5]{Grisvard}.
For  polyhedral domains we have estimate \eqref{H2_Delta} for $p=2$ and convex domains or for $p\geq 6/5$, with $p \neq 2$ and satisfying
\begin{equation} \label{condit_p}
2- 2/p < \pi/\omega, \quad 2- 3/p < \lambda,
\end{equation}
where  $\lambda =\min \{ -1/2 + \sqrt{\lambda_1 + 1/4}, 2\}$,   with $\lambda_1$  the smallest positive eigenvalue  of the Laplace-Beltrami operator on the spherical caps spanning the corners; see e.g.~\cite[Theorem~3.2, Corollary~3.7]{Dauge}. For polyhedra with $\omega \leq \pi/2$  conditions \eqref{condit_p} are satisfied for any $3< p< \infty$,~\cite[Corollary~3.12, Corollary~3.13]{Dauge}.

Thus, the assumptions we made on $\Omega$ ensure that  \eqref{H2_Delta} is satisfied for $p>d$.
We shall also use the Sobolev embeddings, combined with \eqref{H2_Delta} for $p=2$,
\begin{equation}\label{embeding_1}
\begin{aligned}
&\| w\|_{L^\infty(\Omega)} \leq C_{\Omega} \|w \|_{H^2(\Omega)}  \leq C_\Omega \| \Delta w \|_{L^2(\Omega)},
\\
&\|\nabla  w\|_{L^4(\Omega)} \leq C_{\Omega} \|w \|_{H^2(\Omega)}\leq C_\Omega \| \Delta w \|_{L^2(\Omega)},
\end{aligned}
\end{equation}
where by $C_\Omega$ we denote the generic constant in the embedding inequalities, and
\begin{equation}\label{Lipschitz_regul}
 \begin{aligned}
  \|\nabla u \|_{L^\infty} \leq C_\Omega\| u\|_{W^{2,p}}
  \leq C_\Omega \|\Delta u\|_{L^p} \leq C_\Omega \|\Delta u\|_{H^1},
 \end{aligned}
\end{equation}
for $d<p \leq 6$. The first and the last inequalities in  \eqref{Lipschitz_regul} follow from the Sobolev embeddings, whereas  the second inequality is  ensured by \eqref{H2_Delta}.

By $C$ we shall denote a generic constant that is allowed to change from line to line. For  shortness of notation we denote $\| \cdot\|_{L^p(\Omega)}$ by $\|\cdot \|_{L^p}$ and $\| \cdot\|_{H^k(\Omega)}$ by $\|\cdot \|_{H^k}$, with $2\leq p\leq \infty$ and $k=1,2,3$, and 
the $L^2$-inner product is denoted by~$\langle \cdot, \cdot \rangle$.  The semi-norm $\|\nabla \cdot\|_{L^2}$ is denoted by $|\cdot|_{H^1}$.

Consider  
$$
\begin{aligned} 
\mathcal K  = \Big\{& u\in  L^\infty(0,T; H^2(\Omega)) \cap W^{1, \infty}(0,T; H^1_0(\Omega)) \, : \, \\  
& u \in L^\infty(0,T; \dot{H}^3(\Omega)),
\partial_t u \in L^\infty(0,T; H^2(\Omega)), {\partial_t^2 u \in L^2(0,T; H^1(\Omega))},
\\
& \|\Delta u\|_{L^\infty(0,T; L^2(\Omega))} \leq b, \;\;  \|\nabla\Delta u\|_{L^\infty(0,T; L^2(\Omega))}^2 + \kappa\|\Delta\partial_t u\|^2_{L^\infty(0,T; L^2(\Omega))}\leq R^2 \Big\},
\end{aligned} 
 $$
 for some fixed $0<C_\Omega b\leq (1-\kappa)/2k$, with $0<\kappa<1$ and  $C_\Omega$ being the constant in the embedding inequality of $H^2(\Omega)$ in $L^\infty(\Omega)$, and $R^2= C_R \big[ (1+2k C_\Omega b) \|\Delta v_0\|^2_{L^2} + \|\nabla\Delta u_0\|^2_{L^2}\big]$ for some constant $C_R>1$.
 
 The map $\mathcal T: \tilde u \mapsto u =  \mathcal T(\tilde u)$, for  $\tilde u \in \mathcal K$, is defined via the solution of the following linear problem
\begin{equation}\label{main_integro_lin}
\begin{aligned}
  (1-2k\tilde u) \partial_t^2 u -\Delta u - a \beta \ast \partial_t \Delta u  & = 2k \partial_t u \partial_t \tilde u  && \text{in } (0,T) \times \Omega, \\
 u& = 0  && \text{on } (0,T)\times \partial \Omega, \\ 
 u(0)=u_0, \quad \partial_t u(0)& = v_0  && \text{in } \Omega.
 \end{aligned}
\end{equation}
First we show the existence of a unique solution of \eqref{main_integro_lin}. Then by showing that the map $\mathcal T$, for some $T>0$, is a contraction we obtain the existence of a unique solution of \eqref{main_integro}. 

\begin{theorem} \label{existence_linear}
 For $u_0 \in \dot{H}^3(\Omega)$, $v_0 \in H^2(\Omega)\cap H^1_0(\Omega)$ and $\tilde u \in \mathcal K$  there exists a unique solution $u \in L^\infty(0,T;  H^1_0(\Omega))$ of \eqref{main_integro_lin}, with   $u \in L^\infty(0,T; \dot{H}^3(\Omega))$,  $\partial_t u \in L^\infty(0,T; H^2(\Omega))$ and {$\partial_t^2 u \in L^2(0,T; H^1(\Omega))$}.
\end{theorem}
\begin{proof} 
The existence of a unique solution of \eqref{main_integro_lin} can be shown using the Galerkin approximation
$$
u^\ell(t,x)= \sum_{j=1}^\ell c_j^\ell(t) q_j(x), 
$$
where $\{q_j\}_{j\in \mathbb N}$ is a basis of eigenfunctions of $-\Delta$ on $H^1_0(\Omega)$, orthonormal in $L^2$ and orthogonal in $H^1$, with eigenvalues $\{\lambda_j\}$. The coefficient vector  $\bc^\ell = (c_j^\ell)_{j=1}^\ell$ satisfies the following system of ODEs
\begin{equation}\label{ODEs_1}
\frac{d^2}{dt^2} \bc^\ell + \Lambda(t) \bc^\ell + a \Lambda(t) \beta\ast \frac{d}{dt} \bc^\ell-A(t)\frac{d}{dt} \bc^\ell = 0,
\end{equation}
where 
\[
(\Lambda(t))_{ij} = \lambda_j \left\langle \frac1{1-2k\tilde u(\cdot,t)}q_j(\cdot),q_i(\cdot)\right\rangle, \quad
\left(A(t)\right)_{ij} = \left\langle \frac{2k\partial_t \tilde u(\cdot,t)}{1-2k\tilde u(\cdot,t)}q_j(\cdot),q_i(\cdot)\right\rangle.
\]
Writing $\bv^\ell = \frac{d^2}{dt^2} \bc^\ell$ we have that
\begin{equation}\label{eq:c_from_v}
\bc^{\ell}(t) = \int_0^t(t-s)\bv^{\ell}(s)ds + t\frac{d}{dt} \bc^{\ell}(0)+\bc^{\ell}(0)
\end{equation}
and
\[
\frac{d}{dt}\bc^{\ell}(t) = \int_0^t\bv^{\ell}(s)ds + \frac{d}{dt} \bc^{\ell}(0),
\]
with the initial data $\bc^{\ell}(0)$ and $\frac{d}{dt} \bc^{\ell}(0)$ given by the orthogonal projections   onto the basis $\{q_j\}$ of the initial data $u_0$ and $v_0$ respectively.
Thus the original problem \eqref{ODEs_1} is  transformed  to the Volterra integral equation
\[
\bv^\ell(t) + \Lambda(t) \int_0^t (t-s)\bv^\ell(s) ds + a \Lambda(t) \beta_1 \ast \bv^\ell(t)-A(t)\int_0^t\bv^\ell(s)ds = g(t),
\]
where $\beta_1$ is the inverse Laplace transform of
\[
\hat\beta_1(z) := z^{-1} \hat \beta(z)
\]
and
\[
g(t) := -\Lambda(t)\bc^{\ell}(0)+A(t)\frac{d}{dt} \bc^{\ell}(0)-\left(t+a\int_0^t\beta(s)ds\right)\Lambda(t)\frac{d}{dt} \bc^{\ell}(0).
\]
Thus the Volterra integral equation can be written as
\[
\bv^\ell(t) +  \int_0^t K(t,s)\bv^\ell(s) ds = g(t)
\]
with
\[
K(t,s) = \Lambda(t)(t-s+a\beta_1(t-s))-A(t).
\]
From the behaviour of its Laplace transform, we know that $\beta_1$ is analytic for $t >0$ with a singularity of the type $t^\mu$ at $t = 0$, thus the kernel $K(t,s)$ is continuous and so is the right-hand side $g$. The existence of a unique continuous solution follows from \cite[Theorem 2.1.7]{Brunner:2004}. The $C^2[0,T]$-solution $\bc^{\ell}$ of the original problem \eqref{ODEs_1} is then obtained from $\bv^\ell$ and \eqref{eq:c_from_v}.


The existence of a solution of  \eqref{main_integro_lin} is obtained by taking the limit as $\ell\to \infty$ in the Galerkin approximation and using a priori estimates, uniformly in $\ell$, similar to the ones in Lemma~\ref{lem_apriori}. 
\end{proof} 

\begin{lemma} \label{lem_apriori}
 For solution of \eqref{main_integro_lin} we have the following a priori estimates
 \begin{equation}\label{estim_apriori}
 \begin{aligned} 
 & \|\Delta u\|^2_{L^\infty(0,T; L^2(\Omega))} \hspace{-0.1 cm} + \kappa \|\partial_t \nabla u \|^2_{L^\infty(0,T; L^2(\Omega))} 
 \hspace{-0.1 cm} \leq \big[\|\Delta u_0 \|^2_{L^2(\Omega)} + \xi \|\nabla v_0\|^2_{L^2(\Omega)} \big] \times \\
 & \qquad \times \exp\Big\{ T \frac {C_\Omega}\kappa \Big[\|\partial_t \tilde u\|_{L^\infty(\Omega_T)} +  \|\partial_t \Delta \tilde u \|_{L^\infty(0,T; L^2(\Omega))}\\
&\qquad \qquad \quad  + \frac {1}{\kappa}  \|\nabla \tilde u\|_{L^\infty(\Omega_T)}\big( \|\nabla \tilde u\|_{L^\infty(\Omega_T)} + 
 \|\partial_t \nabla\tilde u \|_{L^\infty(0,T;L^2(\Omega))} \big)\Big]\Big\},\\
 & \|\nabla \Delta u \|^2_{L^\infty(0,T;L^2(\Omega))} + \kappa \|\partial_t \Delta u \|^2_{L^\infty(0,T;L^2(\Omega))} \leq \big[\|\nabla \Delta u_0\|^2_{L^2(\Omega)} \\
 &  + \xi \|\Delta v_0\|^2_{L^2(\Omega)} \big]\exp\Big\{T \frac{C_\Omega }\kappa \Big[
\|\partial_t \Delta\tilde u \|_{L^\infty(0,T; L^2(\Omega))} +
\frac 1{\kappa}  \big[\|\Delta \nabla\tilde u \|_{L^\infty(0,T; L^2(\Omega))} \\
& \quad \quad  + \|\nabla \tilde u \|_{L^\infty(\Omega_T)}(1+ \| \Delta \tilde u \|_{L^\infty(0,T; L^2(\Omega))})\big] \big(1 + \|\partial_t\nabla \tilde u \|_{L^\infty(0,T; L^2(\Omega))} \big) \\
&\qquad \qquad  + \frac {1}{\kappa^2} \|\Delta \nabla\tilde u \|_{L^\infty(0,T; L^2(\Omega))}^2 \big(1+ \|\Delta \tilde u \|^2_{L^\infty(0,T; L^2(\Omega))}\big)
   \Big]\Big \}, 
   \end{aligned}
\end{equation}
together with 
\begin{equation}\label{estim_apriori_1}
 \begin{aligned}
   \|\beta \ast \Delta \partial_t u \|^2_{L^2(\Omega_T)}
  \leq  \big[\|\Delta u_0 \|^2_{L^2(\Omega)} + \xi \|\nabla v_0\|^2_{L^2(\Omega)} \big] \Big(T\frac{ C_\Omega}{\kappa} \Big[\frac 1{ \kappa} \|\nabla \tilde u \|_{L^\infty(\Omega_T)} \quad
   \\
    + \|\partial_t \Delta \tilde u \|_{L^\infty(0,T; L^2(\Omega))} \Big] \Big[ 1 + \frac 1\kappa \|\nabla \tilde u\|_{L^\infty(\Omega_T)} \Big]  \exp\Big\{ T\frac{ C_\Omega}{\kappa}  \Big[\frac 1{ \kappa} \|\nabla \tilde u \|_{L^\infty(\Omega_T)}\\
    + \|\partial_t \Delta \tilde u \|_{L^\infty(0,T; L^2(\Omega))} \Big] \Big[ 1 + \frac 1\kappa \|\nabla \tilde u\|_{L^\infty(\Omega_T)} \Big]\Big\}+ \frac 1 a\Big),
 \end{aligned}
\end{equation}
and 
\begin{equation*}
 \begin{aligned}
&\kappa^2 \|\partial_t^2 \nabla  u \|^2_{L^2(\Omega_T)}+  \|\beta \ast \Delta \nabla \partial_t u \|^2_{L^2(\Omega_T)}
\leq
 \big[\|\nabla \Delta u_0\|^2_{L^2(\Omega)} + \xi \|\Delta v_0\|^2_{L^2(\Omega)} \big] \times \\
 &  \times \Big[T C_\Omega \Big( 1+ \frac 1\kappa \Big[ \|\partial_t \Delta\tilde u \|_{L^\infty(0,T; L^2(\Omega))}  +
\frac 1{\kappa}  \big[ \|\nabla \tilde u \|_{L^\infty(\Omega_T)}(1+ \| \Delta \tilde u \|_{L^\infty(0,T; L^2(\Omega))})\\
&  + \|\Delta \nabla\tilde u \|_{L^\infty(0,T; L^2(\Omega))}\big]  \big[1  + \|\partial_t\nabla \tilde u \|_{L^\infty(0,T; L^2(\Omega))} \big]   + \frac {1}{\kappa^2} \|\Delta \nabla\tilde u \|_{L^\infty(0,T; L^2(\Omega))}^2 \big[1 \\
&  + \|\Delta \tilde u \|^2_{L^\infty(0,T; L^2(\Omega))}\big]
   \Big] \Big)   \exp\Big\{T \frac{C_\Omega }\kappa \Big[
\|\partial_t \Delta\tilde u \|_{L^\infty(0,T; L^2(\Omega))} +
\frac 1{\kappa}  \big[\|\Delta \nabla\tilde u \|_{L^\infty(0,T; L^2(\Omega))} \\
& \qquad \quad  + \|\nabla \tilde u \|_{L^\infty(\Omega_T)}\big(1+ \| \Delta \tilde u \|_{L^\infty(0,T; L^2(\Omega))}\big)\big] \big(1 + \|\partial_t\nabla \tilde u \|_{L^\infty(0,T; L^2(\Omega))} \big) \\
&\qquad \qquad  + \frac {1}{\kappa^2} \|\Delta \nabla\tilde u \|_{L^\infty(0,T; L^2(\Omega))}^2 \big(1+ \|\Delta \tilde u \|^2_{L^\infty(0,T; L^2(\Omega))}\big) 
   \Big]\Big \}+\frac 1 a \Big],
 \end{aligned}
\end{equation*}
where $\xi = 1+ 2k C_\Omega b$,  $\tilde u \in \mathcal{K}$, $b$ and $\kappa$ as in the definition of $\mathcal{K}$, $\Omega_T = (0,T)\times \Omega$,  and  the constant $C_\Omega>0$  includes constants from the embedding inequalities and hence depends on the domain $\Omega$.
\end{lemma}

\begin{proof} 
Considering $\partial_t u$ as a test function in the weak formulation of~\eqref{main_integro_lin} yields 
$$
\begin{aligned} 
& \kappa\|\partial_t u \|^2_{L^\infty(0,T; L^2(\Omega))} + \|\nabla u \|^2_{L^\infty(0,T; L^2(\Omega))}   \\
&  \leq  \big[(1 + 2kC_\Omega b)\|v_0 \|^2_{L^2(\Omega)} + \|\nabla u_0\|^2_{L^2(\Omega)} \big]  \exp\Big\{ T \frac{2k}\kappa  \|\partial_t \tilde u\|_{L^\infty(\Omega_T)}\Big\} ,
\end{aligned} 
$$
 where $1- 2k \|\tilde u \|_{L^\infty((0,T)\times\Omega))} \geq 1- 2kC_\Omega b \geq \kappa$ and we used Lemma~\ref{lem:time_lower} in the simpler form $\int_0^t \beta \ast \partial_t \nabla u \cdot \partial_t\nabla u \,d\tau \geq 0$.

 Considering $-\Delta \partial_t u$ as a test function for \eqref{main_integro_lin} and using estimates in  Lemma~\ref{lem:time_lower}  we obtain 
$$
\begin{aligned} 
& \kappa\|\partial_t \nabla u \|^2_{L^\infty(0,T; L^2(\Omega))} + \|\Delta u \|^2_{L^\infty(0,T; L^2(\Omega))}    \leq  \big[\xi\|\nabla v_0 \|^2_{L^2(\Omega)} + \|\Delta u_0\|^2_{L^2(\Omega)} \big] \times \\
& \qquad \times \exp\Big\{ T \frac {C_\Omega}\kappa  \Big[\|\partial_t \tilde u\|_{L^\infty(\Omega_T)} + \|\partial_t \nabla \tilde u \|_{L^\infty(0,T; L^4(\Omega))}\\
&\qquad \qquad + \frac {1}{\kappa}  \|\nabla \tilde u\|_{L^\infty(\Omega_T)}\Big(1 + 
  \|\partial_t \tilde u \|_{L^\infty(0,T;L^4(\Omega))} + \frac{1}{\kappa} \|\nabla \tilde u\|_{L^\infty(\Omega_T)} \Big)\Big]\Big\} 
\end{aligned} 
$$
and 
$$
\begin{aligned} 
& \|\beta \ast \Delta \partial_t u \|^2_{L^2(\Omega_T)} 
\leq \big[\xi\|\nabla v_0 \|^2_{L^2(\Omega)} + \|\Delta u_0\|^2_{L^2(\Omega)} \big] \Big[T
\frac {C_\Omega}\kappa  \Big(  \|\partial_t \nabla \tilde u \|_{L^\infty(0,T; L^4(\Omega))} 
\\
& \; + \|\partial_t \tilde u\|_{L^\infty(\Omega_T)}+ \frac 1{\kappa} \|\nabla \tilde u\|_{L^\infty(\Omega_T)}\big[1 + 
  \|\partial_t \tilde u \|_{L^\infty(0,T;L^4(\Omega))} +  \frac{1}{\kappa} \|\nabla \tilde u\|_{L^\infty(\Omega_T)} \big] \Big)\times \\
& \quad \times \exp\Big\{ T \frac {C_\Omega}\kappa  \Big(\|\partial_t \tilde u\|_{L^\infty(\Omega_T)} +  \|\partial_t \nabla \tilde u \|_{L^\infty(0,T; L^4(\Omega))}\\
&\quad + \frac 1{\kappa} \|\nabla \tilde u\|_{L^\infty(\Omega_T)}\big[1 +   \|\partial_t \tilde u \|_{L^\infty(0,T;L^4(\Omega))}+ \frac{1}{\kappa} \|\nabla \tilde u\|_{L^\infty(\Omega_T)} \big]\Big)\Big\} + \frac 1 a \Big].
\end{aligned} 
$$
Applying $\Delta$ to  \eqref{main_integro_lin} and taking $\Delta \partial_t u$ as a test function in the weak formulation of the problem  implies 
$$
\begin{aligned} 
& \kappa\|\partial_t \Delta u \|^2_{L^\infty(0,T; L^2(\Omega))} + \|\Delta \nabla u \|^2_{L^\infty(0,T; L^2(\Omega))}   \leq \big[  \xi \|\Delta v_0\|^2_{L^2(\Omega)} + \|\Delta \nabla u_0\|^2_{L^2(\Omega)}\big]\times \\
& \times \exp\Big\{T \frac{C_\Omega} \kappa\Big(  
\|\partial_t \Delta\tilde u \|_{L^\infty(0,T; L^2(\Omega))}  + \frac 1{\kappa}  \Big[\|\Delta \tilde u \|_{L^\infty(0,T; L^4(\Omega))} \big(\|\partial_t \nabla \tilde u \|_{L^\infty(0,T;  L^2(\Omega))}   
\\
&  
 \qquad \quad + 1 \big) + \|\nabla \tilde u \|_{L^\infty(\Omega_T)} \big( \|\nabla\tilde u\|_{L^\infty(0,T;L^4(\Omega))}+ 1\big) \big(\|\partial_t\nabla \tilde u \|_{L^\infty(0,T;  L^2(\Omega))}+1\big) \Big]
 \\ 
 & \qquad \quad 
   + \frac {1}{\kappa^2} \Big[\|\Delta \tilde u \|_{L^\infty(0,T; L^4(\Omega))}^2+ \|\nabla \tilde u \|^2_{L^\infty(\Omega_T)}\big(1+ \|\nabla \tilde u \|^2_{L^\infty(0,T; L^4(\Omega))}\big) \Big]
 \Big)\Big \}. 
\end{aligned} 
$$
Using  the Sobolev embedding inequality yields the second estimate in~\eqref{estim_apriori}. 
From those estimates, using the weak formulation of the problem,  we also obtain the estimate for $\beta \ast \Delta \nabla \partial_t u$ in $L^2((0,T)\times \Omega)$.
The strong formulation of the equation in~\eqref{main_integro_lin}, see~\eqref{second_deriv_strong_1} in Appendix, together with the estimates for $\nabla \Delta u$ in $L^\infty(0,T; L^2(\Omega))$, $\partial_t u$ in $L^\infty(0,T; H^2(\Omega))$, and $\beta \ast \Delta \nabla \partial_t u$ in $L^2((0,T)\times \Omega)$, implies  the estimate
for $\partial_t^2 \nabla u$ in $L^2((0,T)\times \Omega)$.
See appendix for more details on the derivation of a priori estimates.
\end{proof}

\begin{remark}\label{rem:testing}
For simplicity of presentation we have skipped the Galerkin approximation step in the above proof. Note that in the Galerkin approximation, using the notation from Theorem~\ref{existence_linear}, $\partial_t \Delta u^{\ell}$ satisfies the zero Dirichlet boundary condition and hence boundary integrals vanish when integrating by parts. Notice that  the limit as $\ell \to \infty$ in $H^1$-norm  of the Galerkin approximation $\Delta u^{\ell}$   yields $\Delta u = 0$ on $\partial \Omega$ and in the estimates in Lemma~\ref{lem_apriori}  the equivalence between  the $H^1$-norm of $\Delta u$ and  the semi-norm $\|\nabla \Delta u\|_{L^2}$ is used.
\end{remark}

Using a priori estimates proven in Lemma~\ref{lem_apriori} and applying the Banach fixed point theorem yield  local existence of a unique solution of nonlinear problem~\eqref{main_integro}.
\begin{theorem} \label{thm:main_cont}
 For $u_0 \in \dot{H}^3(\Omega)$ and  $v_0 \in H^2(\Omega)\cap H^1_0(\Omega)$, with
 $$\|\Delta u_0\|^2_{L^2(\Omega)} + (1+ 2k C_\Omega b)\|\nabla v_0 \|^2_{L^2(\Omega)}\leq \eta b^2, $$ for any
  $\eta \in (0,1)$, there exists time interval $T=T(R, b, \eta) > 0$ such that $u \in \mathcal K$  is a unique solution of \eqref{main_integro}.
\end{theorem}
\begin{proof}
For $R^2=C_R  \big[\|\nabla \Delta u_0\|^2_{L^2(\Omega)} + (1+2k C_\Omega b) \|\Delta v_0\|^2_{L^2(\Omega)}\big]$ and $T=T(R, b, \eta)$ such that
$$
\exp\{T C_\Omega (R+ R^2)\} \leq 1/\eta \; \; \text{  and } \;\;   \exp\{T C_\Omega (R+ R^2+R^3+ R^4)\} \leq  C_R,
$$
estimates in \eqref{estim_apriori} imply $u= \mathcal T(\tilde u) \in \mathcal{K}$ for $\tilde u \in \mathcal{K}$.

To show that $\mathcal T \colon \mathcal K \to \mathcal K$ is a contraction we consider \eqref{main_integro_lin} for $\tilde u_1$ and $\tilde u_2$ in ${\mathcal K}$
and, taking $-\Delta \partial_t (u_1 - u_2)$ as a test function for 
the difference of the corresponding equations, obtain
$$
\begin{aligned} 
& \kappa  \|\nabla \partial_t (u_1 - u_2) \|^2_{L^\infty(0,T; L^2(\Omega))} + 
 \|\Delta (u_1 - u_2) \|^2_{L^\infty(0,T; L^2(\Omega))}  \\
& \leq 
   C_\Omega\Big(
 \Big[\|\partial_t \Delta u_2 \|^2_{L^2(\Omega_T)} + 
 \frac 1 \kappa \|\partial_t \nabla u_2 \|^2_{L^2(\Omega_T)}  \Big]
 \|\nabla \partial_t (\tilde u_1 - \tilde u_2) \|^2_{L^\infty(0,T; L^2(\Omega))}
 \\
 & + \frac 1 \kappa\Big[\|\Delta \nabla u_2\|^2_{L^2(\Omega_T)}  +  \|\partial_t \nabla \tilde u_2 \|^2_{L^2(0,T; H^1(\Omega))} + 
 \|\Delta u_2\|^2_{H^1(0,T; L^2(\Omega))} \\
 & + \|\beta \ast \Delta \partial_t u_2\|_{L^1(0,T; H^1(\Omega))}\big(1+ \|\nabla \tilde u_2\|_{L^\infty(\Omega_T)}\big)\Big]
 \|\tilde u_1 - \tilde u_2\|_{L^\infty(0,T;H^2(\Omega))}^2 \Big)\times 
 \\
 & \times
 \exp\Big\{  \frac {C_\Omega}{ \kappa } \Big(\|\partial_t \Delta \tilde u_1\|_{L^1(0,T;L^2(\Omega))} 
 +  
  \frac 1 \kappa \Big[
 \|\nabla \tilde u_2\|^2_{L^\infty(\Omega_T)} \|\partial_t \nabla u_2 \|^2_{L^2(\Omega_T)} \\
 & + \|\beta \ast \Delta \partial_t u_2\|_{L^1(0,T; H^1(\Omega))}(1+ \|\nabla \tilde u_2\|_{L^\infty(\Omega_T)})
 + T\big(1+ \|\nabla \tilde u_1\|^2_{L^\infty(\Omega_T)}
 \\
 & + \| \tilde u_2\|^2_{W^{1,\infty}(\Omega_T)}+ \|\partial_t  u_2\|^2_{L^\infty(\Omega_T)}\big) + \|\nabla \tilde u_1\|_{ L^\infty(\Omega_T)} \|\partial_t \nabla \tilde u_1 \|_{L^1(0,T; L^2(\Omega))} \Big] \\
 & + 
 \frac 1 {\kappa^2} \|\nabla \tilde u_1\|^2_{L^2(0,T; L^\infty(\Omega))} 
 \Big) \Big\}
 \leq C_\Omega [T R^2+ T^{\frac 12}(R+R^2)] \times \\
 & \times \exp\big\{C_\Omega\big[T (1+R + R^2+  R^4) +  T^{\frac 12} (R^3+ R^4) \exp\{ C_\Omega T (R+ R^2)\}\big]\big\}\times \\
& \qquad \times \big( \kappa \|\nabla \partial_t (\tilde u_1 - \tilde u_2) \|^2_{L^\infty(0,T; L^2(\Omega))} + 
  \|\Delta(\tilde u_1 - \tilde u_2)\|_{L^\infty(0,T;L^2(\Omega))}^2\big).
\end{aligned}
$$
Then for $T$ such that $ C_\Omega [T R^2+ T^{\frac 12}(R+R^2)] \exp\big\{C_\Omega\big[T (1+R + R^2+  R^4) +  T^{\frac 12} (R^3+ R^4) \exp\{ C_\Omega T (R+ R^2)\}\big]\big\} <1$ we have that $\mathcal T : \mathcal{K} \to \mathcal{K}$ is a contraction. Thus applying the Banach fixed point theorem, and iterating over time,  yields  existence of a unique solution of the nonlinear problem~\eqref{main_integro}. 
%
\end{proof}


\section{Trapezoidal discretization}
In this section we present analysis for the numerical scheme for problem~\eqref{main_integro}. The time semi-discretization considered here is based on trapezoidal time-stepping with uniform time-step $\tstep >0$ and  $n = 1,2 \dots, N$, with {$T=N \tstep$},
\begin{equation} \label{main_discrete}
(1-2k\{u\}_n)D^2u_n-\Delta \{u\}_{n}-a \beta\ast_{\tstep} D\Delta u_n = 2k(Du_n)^2,
\end{equation} 
where $u_n \in H^1_0(\Omega)\cap H^2(\Omega)$ and 
\[
\begin{aligned}
& Du_n = \frac1{2\tstep}(u_{n+1}-u_{n-1}), \;&&
D^2u_n = \frac1{\tstep^2}(u_{n+1}-2u_n+u_{n-1}), \; \\
& \{u\}_n = \frac14(u_{n+1}+2u_n+u_{n-1}), \; && \tilde D u_n = \frac1{\tstep}(u_{n+1}-u_{n}),
\end{aligned}
\]
with  $Du_0: = v_0$,  and  $[\beta \ast_{\tstep} g]_n$ (with the square brackets in most places left-out) a convolution quadrature approximation  of   $\int_0^{t_n}\beta(t_n-\tau)g(\tau) d\tau$. 
We will use convolution quadrature based on the second order backward difference formula (BDF2) \cite{Lubich1986,Lubich1988I,Lubich1988II}  which results in the discrete convolution 
\[
[\beta\ast_{\tstep} v ]_n = \sum_{j = 0}^n \omega_{n-j} v_j, 
\]
with convolution weights $\omega_j$ given by the generating function
\[
\hat\beta\left(\frac{\delta(\zeta)}{\tstep}\right) = \sum_{j = 0}^\infty \omega_j \zeta^j, \qquad \delta(\zeta) = (1-\zeta)+\frac12(1-\zeta)^2.
\]
For $v$ that is sufficiently smooth and with sufficiently many zero derivatives at $t =0 $, we have that $[\beta\ast_{\tstep} v ]_n = \beta \ast v(t_n)+ \O(\tstep^2)$, whereas for $v(t) = t^\alpha$ and real $\alpha > -1$ 
\begin{equation}
  \label{eq:CQ_poly_error}
\Big|[\beta\ast_{\tstep} v ]_n-\beta \ast v(t_n)\Big|
\leq 
    \begin{cases} 
Ct_n^{\mu -1} \tstep^{\alpha+1}      & \text{for } -1 < \alpha \leq 1,\\ 
Ct_n^{\mu +\alpha-2} \tstep^{2}  & \text{for } \alpha \geq 1,
\end{cases} 
\end{equation}
for $n = 1,\dots$; see \cite[Theorem 2.2]{Lubrev}. 

Alternatively, we can use the corrected CQ formula
\begin{equation}
  \label{eq:corrected}
  [\beta \tilde \ast_{\tstep} v]_n := [\beta \ast_{\tstep} v]_n+\omega_{n,0}v_0,
\end{equation}
where $\omega_{n,0}$ is chosen so that the formula is exact for constant function, i.e., 
\[
\omega_{n,0} := \int_0^{t_n} \beta(\tau) d\tau -[\beta \ast_\tstep 1]_n =  \int_0^{t_n} \beta(\tau) d\tau -\sum_{j = 0}^n \omega_j.
\]
Note the trivial but useful fact that  $\beta \tilde \ast_{\tstep} v  \equiv \beta \ast_{\tstep} v$ if $v_0 = 0$. From \eqref{eq:CQ_poly_error} and the definition of $\omega_{n,0}$ it follows that we have the stability bound
\begin{equation}
  \label{eq:corr_bnd}
  |\omega_{n,0}| \leq C t_n^{\mu -1} \tstep
\end{equation}
for $n \geq 1$. The first correction weight is $\omega_{0,0} = -\omega_0$, where
$\omega_0 = \hat \beta(\delta(0)/\tstep) = \hat\beta\left(\frac3{2\tstep}\right) \sim (2/3)^\mu \tstep^\mu$ as $\tstep \to 0$. In the estimates below, we will only require the fact that  $\omega_{n,0}$ are bounded by a constant independent of $\tstep$ for all $n \geq 0$. The semi-discretisation with the corrected CQ formula reads
\begin{equation} \label{main_discrete_corr}
(1-2k\{u\}_n)D^2u_n-\Delta \{u\}_{n}-a \beta\tilde \ast_{\tstep} D\Delta u_n = 2k(Du_n)^2.
\end{equation} 
 When using the corrected scheme, we will further assume that $v_0 \in H^3(\Omega)$.

A crucial property of (the non-corrected) convolution quadrature, see \cite[Lemma~2.1]{BaLuSa:2015} and \cite[Theorem~2.25]{lbbook}, is that \eqref{eq:hatbeta_lower} implies
\begin{equation}
  \label{eq:CQ_pos}
  \sum_{j = 0}^\infty \scl^{2j} \left\langle v_j, [\beta \ast_\tstep v]_j\right\rangle
\geq C_\beta \sum_{j = 0}^\infty \scl^{2j} \|[\beta \ast_\tstep v]_j\|_{L^2}^2,
\end{equation}
for $\scl = e^{-\sigma \tstep}$, with $\sigma > 0$, and 
\[
C_\beta := (\tilde\sigma+r)^\mu \; \;(\text{if } \beta =\betaA), \quad C_\beta := 1 \;  \;(\text{if } \beta =\betaB),
\]
where $\tilde \sigma = C_2\min(1,\sigma)$. 
Thus, if $\beta = \betaA$  with $r > 0$ or $\beta = \betaB$, we can set $\sigma = 0$ (i.e., $\scl = 1$ and $\tilde \sigma  =0 $) and still obtain positivity of the left hand side in \eqref{eq:CQ_pos}. 

For the corrected version, all we can say is that
\begin{equation}
  \label{eq:CQ_pos_corr}
  \sum_{j = 0}^\infty \scl^{2j} \left\langle v_j, [\beta \tilde \ast_\tstep v]_j\right\rangle
\geq C_\beta \sum_{j = 0}^\infty \scl^{2j} \|[\beta \ast_\tstep v]_j\|_{L^2}^2+ \sum_{j = 0}^\infty \scl^{2j} \omega_{j,0}\left\langle v_j,v_0\right\rangle.
\end{equation}

To initiate the iterations in~\eqref{main_discrete}  we set 
\begin{equation}
  \label{eq:initial}
u_1 = u_0+ \tstep v_0+\frac12 \tstep^2 \partial_t^2u(0),  
\end{equation}
where we can determine $\partial_t^2u(0)$ from the equation \eqref{main_integro}
\begin{equation}
  \label{eq:dttu0}
    \partial_t^2 u(0) = \frac1{1-2ku_0}\left(\Delta u_0 + 2k (v_0)^2\right).
\end{equation}

\begin{lemma} \label{lem:disc_bounds}
 Under the assumptions on the initial data \eqref{init_cond_assum}, {along with $u_0 \in H^4(\Omega)$} and
$$
 \Big\|\Delta\frac{ u_0+ u_1} 2\Big\|^2_{L^2(\Omega)} + (1+ 2k C_\Omega b) \Big\|\nabla \frac{u_1 -u_0}{\tstep} \Big\|^2_{L^2(\Omega)} \leq \eta b^2,
$$
where $0< C_\Omega b \leq (1- \kappa)/2k$, with $0< \kappa <1$, and the constant $C_\Omega$ is the constant in \eqref{H2_Delta} and \eqref{embeding_1},
  we have the following stability estimates for the scheme~\eqref{main_discrete}
\begin{equation} \label{eq:stab_est}
 \begin{aligned} 
  &  \kappa \sup_{1\leq n\leq N-1}\|\tilde D \nabla u_n \|^2_{L^2(\Omega)} + \sup_{1\leq n\leq N-1}\|\Delta (u)_n\|^2_{L^2(\Omega)}\leq b^2, \\
  &   \kappa \sup_{1\leq n\leq N-1} \|\tilde D \Delta u_n \|_{L^2(\Omega)}^2 + \sup_{1\leq n\leq N-1}\|\nabla \Delta(u)_n\|^2_{L^2(\Omega)} \leq R^2,
 \end{aligned}
\end{equation}
where  $(u)_n = (u_{n+1} + u_n)/2$ and  $R^2 = C_R\big[ \| \Delta \nabla (u)_0 \|_{L^2}^2 + (1+2kC_\Omega b)\|\tilde D \Delta u_0 \|_{L^2}^2 \big]$ with $C_R >1$.
\end{lemma}

\begin{proof}
In order to analyse the system we need that $1-2k\{u\}_n\geq \kappa >0$ for some (fixed) $\kappa \in (0,1)$. Thus similarly to the continuous case, we consider the fixed-point iteration 
\begin{equation} \label{main_discrete_k}
(1-2k  d_n)D^2u_n-\Delta \{u\}_{n}-a\beta\ast_{\tstep} D\Delta u_n = 2k  v_n Du_n,
\end{equation}
where $d_n =\{\tilde u\}_n$ and $v_n = D\tilde u_n$ for $\tilde u_n \in \dot H^3(\Omega) \cap H^1_0(\Omega)$ satisfying  \eqref{eq:stab_est} (with $u_n$ replaced by $\tilde u_n$).


To derive the stability estimates we first test \eqref{main_discrete_k} with $\scl^{2n}Du_n$,  $\scl = e^{-\tstep/T}$,  and estimate each term separately.  For the first term  we have 
\[
  \begin{split}
   &\tstep\sum_{n = 1}^{N-1}\scl^{2n}\big\langle(1-2k d_n)D^2u_n,D u_n\big\rangle \\
\geq & \frac12 \int_\Omega \scl^{2(N-1)}(1-2k  d_{N-1})\left(\frac{u_N-u_{N-1}}{\tstep}\right)^2dx - \frac12 \int_{\Omega}(1-2k d_1)\left(\frac{u_1-u_0}{\tstep}\right)^2dx\\
&+k \int_\Omega \sum_{n = 1}^{N-2} \scl^{2(n+1)} (d_{n+1}-d_n) \Big(\frac{u_{n+1}-u_n}{\tstep}\Big)^2dx\\
= & \frac12 \int_\Omega \scl^{2(N-1)}(1-2k d_{N-1})\Big(\frac{u_N-u_{N-1}}{\tstep}\Big)^2dx - \frac12 \int_{\Omega}(1-2k d_1)\Big(\frac{u_1-u_0}{\tstep}\Big)^2dx\\
&+k\tstep \int_\Omega \sum_{n = 1}^{N-2} \scl^{2(n+1)} \left(\frac{(\tilde u)_{n+1}-(\tilde u)_{n-1}}{2\tstep}\right)\big(\tilde D u_{n}\big)^2dx.
\end{split}
\]
The last term in the estimate above can be bounded by
\[
k \tstep \sum_{n = 1}^{N-2} \scl^{2(n+1)} \big\|D(\tilde u)_n\big\|_{L^\infty} \big\|\tilde Du_n\big\|^2_{L^2}. 
\]
Using  $1-2kd_n \geq \kappa >0$,   together with  the estimates for the convolution quadrature, yields 
\[
  \begin{split}    
E_{N-1} &\leq E_0-C_\beta \tstep \sum_{n = 1}^{N-1}\varrho^{2n}\big|\beta\ast_{\tstep} Du_n\big|^2_{H^1}\\&
\quad + k\tstep\sum_{n=1}^{N-2}  \scl^{2(n+1)} \big\|D(\tilde u)_n\big\|_{L^\infty}\big\|\tilde Du_n\big\|^2_{L^2} + 2k\tstep\sum_{n=1}^{N-1}  \scl^{2n} \big\|D\tilde u_n\big\|_{L^\infty}\big\|Du_n\big\|^2_{L^2} \\
& \leq E_0-C_\beta \tstep \sum_{n = 1}^{N-1}\varrho^{2n}\big|\beta\ast_{\tstep} Du_n\big|^2_{H^1}
+ C \sup\limits_{0\leq n \leq N-1} \big\|\tilde D\tilde u_n\big\|_{L^\infty}  \tstep\sum_{n=0}^{N-1}  \scl^{2n} \big\|\tilde Du_n\big\|^2_{L^2},
  \end{split}
\]
where for $n \geq 1$
\[
E_n = \frac12 \scl^{2n}\left|\frac{u_{n+1}+u_{n}}{2}\right|_{H^1}^2+ \frac12 \int_\Omega \scl^{2n}(1-2k d_{n})\big |\tilde D u_n\big |^2dx
\]
and
\[
  E_0 = \frac12 \left|\frac{u_1+u_0}{2}\right|_{H^1}^2+  \frac12 \int_{\Omega}(1-2kd_1)\big|\tilde D u_0\big|^2 dx. 
\]
Then the discrete Gr\"onwall inequality ensures
\[
E_{N-1} \leq  E_0 \exp\Big\{C   \sup\limits_{0\leq n \leq N-1}\big\|\tilde D\tilde u_n\big\|_{L^\infty} \tstep\sum_{n=0}^{N-1}  \scl^{2n}\left\|\frac1{1-2kd_n}\right\|_{L^\infty} \Big\}.
\]

When considering the corrected convolution quadrature, we will have the aditional term 
\begin{equation}\label{corrected_estim_v0}
\begin{aligned}
&\tstep \sum_{n=1}^{N-1}\scl^{2n} \big| \omega_{n,0} \langle \nabla Du_n, \nabla Du_0 \rangle  \big| \\
&\leq  C  \tstep \sum_{n=1}^{N-1}\scl^{2n}  \Big[ \frac \delta 2\big( \|\nabla (u)_{n+1}\|^2_{L^2} + \|\nabla (u)_{n}\|^2_{L^2} \big) + C_\delta \| v_0\|_{H^1}^2 \Big]\\
& \leq  \delta \sup_n \|\nabla (u)_{n}\|^2_{L^2} +C, 
\end{aligned} 
\end{equation}
for any fixed $\delta >0$. 
Then the first term can be subtracted from the corresponding term on the left-hand side.

Taking $-\scl^{2n} \Delta Du_n$ as a test function in \eqref{main_discrete_k} and integrating by parts in the first term on the left-hand side and in the right-hand side  yield
\begin{equation}\label{Grad_eq}
 \begin{aligned} 
 \tstep \sum_{n=1}^{N-1} \scl^{2n} \Big[\big\langle  (1- 2k d_n) D^2 \nabla u_n, D\nabla u_n \big\rangle + \big\langle \Delta  \{u\}_{n}, \Delta  D u_n \big\rangle 
\\
+  \Big\langle a \beta \ast_{\tstep} D \Delta    u_n,   D \Delta  u_n \Big\rangle \Big]
\\
=2k\tstep \sum_{n=1}^{N-1} \scl^{2n}\big\langle Du_n  \nabla  v_n 
+ D\nabla u_n  v_n +  \nabla d_n D^2 u_n, D\nabla u_n \big\rangle  . 
 \end{aligned}
\end{equation} 
Using equation \eqref{main_discrete_k} we rewrite the last term on the right-hand side  as
$$
\begin{aligned} 
 \big\langle \nabla d_n D^2 u_n, D\nabla u_n \big\rangle
= \Big\langle \frac 1 { 1 -2 k d_n}\nabla d_n\Big( \Delta\{ u\}_n + a \beta\ast_{\tstep} D\Delta  u_n\Big), D\nabla u_n \Big\rangle 
\\
 + 
2k\Big \langle \frac 1 { 1 -2 k d_n} \nabla d_n D u_n v_n, D\nabla u_n \Big \rangle. 
\end{aligned} 
$$
Similar as above for the first term in \eqref{Grad_eq} we have 
$$
\begin{aligned} 
 & \tstep \sum_{n=1}^{N-1} \scl^{2n} \Big[\big\langle  (1- 2k d_n) D^2 \nabla u_n, D\nabla u_n \big \rangle 
 \\
& \quad  \geq \frac12 \scl^{2(N-1)} \int_\Omega \big(1-2k d_{N-1}\big)\big|\nabla \tilde D u_{N-1}\big|^2dx 
 - \frac12 \int_{\Omega}\big(1-2k d_{1}\big)\big|\nabla \tilde D u_0\big|^2dx \\
  & \qquad +k\tstep\sum_{n = 1}^{N-2}\scl^{2(n+1)}\int_\Omega  D(\tilde u)_{n} \big|\nabla \tilde D u_n\big|^2 dx.
\end{aligned}
$$
For the second and third terms in \eqref{Grad_eq} we have 
$$
 \begin{aligned} 
 & \tstep \sum_{n=1}^{N-1} \scl^{2n} \Big[\big\langle \Delta  \{u \}_{n}, \Delta  D u_n \big\rangle 
 + \big\langle a \beta \ast_{\tstep} D\Delta    u_n,   D \Delta   u_n \big\rangle \Big] 
\\
&\quad  \geq \frac 12 \scl^{2(N-1)}\Big \| \frac{ \Delta  (u_N + u_{N-1})} 2 \Big \|_{L^2}^2 - 
\frac 12 \Big \| \frac{ \Delta  (u_1 + u_{0})} 2 \Big \|_{L^2}^2
 \\
 & \qquad + C_\beta \tstep \sum_{n=1}^{N-1} \scl^{2n} \big\|\beta\ast_{\tstep} D \Delta   u_n \big\|_{L^2}^2 .
 \end{aligned}
$$
The terms on the right-hand side are estimated as 
$$
\begin{aligned} 
 & \tstep \sum_{n=1}^{N-1} \scl^{2n}
 \Big[\Big\langle \frac {\nabla d_n} { 1 -2 k d_n} \Delta\{ u\}_n, D\nabla u_n \Big\rangle 
\\
&+ a\Big\langle \frac {\nabla d_n} { 1 -2 k d_n}  \beta \ast_{\tstep} D\Delta  u_n, D\nabla u_n \Big\rangle + 
2k \Big\langle \frac {\nabla d_n} { 1 -2 k d_n}  D u_n v_n, D\nabla u_n \Big\rangle \Big]
\\
& \leq  \varsigma \tstep \sum_{n=1}^{N-1} \scl^{2n} \|\beta\ast_{\tstep} D\Delta  u_n\|_{L^2}^2 + 
C \tstep \sum_{n=1}^{N-1} \scl^{2n} \Big\|\frac 1 { 1 -2 k d_n}\Big\|_{L^\infty} \|\nabla d_n \|_{L^\infty} \times \\
&\quad  \times \Big[ \|\Delta \{ u \}_n \|_{L^2}^2 
+\|D\nabla u_n \|_{L^2}^2  \Big(1 +  \|v_n\|_{L^4} + C \Big\|\frac 1 { 1 -2 k d_n}\Big\|_{L^\infty} \|\nabla d_n \|_{L^\infty}  \Big)
  \Big],  
\end{aligned}
$$
where we assume $\varsigma \leq C_\beta/2$, and
$$
\begin{aligned} 
2k\tstep \sum_{n=1}^{N-1} \scl^{2n} \Big[\big\langle  Du_n \nabla  v_n,  D\nabla u_n \big\rangle  + \big\langle  D\nabla u_n v_n,  D\nabla u_n \big\rangle \Big]
\\
\leq C\tstep \sum_{n=1}^{N-1} \scl^{2n}
\big(\| v_n \|_{L^\infty} + \|\nabla v_n\|_{L^4}\big) \|D \nabla  u_n\|_{L^2}^2 .
\end{aligned}
$$
Applying the discrete Gr\"onwall inequality we obtain 
\begin{equation} 
 \begin{aligned} 
  & \int_\Omega \scl^{2(N-1)}(1-2kd_{N-1})|\tilde D \nabla u_{N-1}|^2dx + 
  \scl^{2(N-1)} \|\Delta (u)_{N-1} \|^2_{L^2} \\
 & +  C_\beta\tstep \sum_{n = 1}^{N-1} \scl^{2n}\left\|\beta\ast_{\tstep} D \Delta u_n\right\|^2_{L^2} \leq \big[\nu\|\tilde D u_0 \|^2_{L^2}+ \big\|\Delta (u)_0 \big\|^2_{L^2}   \big]\times \\
 &\quad \times \exp\Big\{ C\tstep\sum_{n = 1}^{N-1} \Big\|\frac 1{1 -2 k d_{n}}\Big \|_{L^\infty} \Big[\|\nabla v_n \|_{L^4} + \| v_n \|_{\infty}  \\
& \qquad + 
\frac{\big\|  \nabla d_{n} \big\|_{L^\infty}}{\big\|{1 -2 k d_{n}}\big \|_{L^\infty}} \Big(\Big \|\frac 1{1 -2 k d_{n}}\Big \|_{L^\infty} \big\|  \nabla d_{n} \big\|_{L^\infty} 
+ 1+ \|v_n\|_{L^4}\Big) \Big]\\
& \leq \Big[   \nu\|\nabla\tilde D u_0 \|^2_{L^2} + \|\Delta (u_0+ u_1)/2\|^2_{L^2} \Big]  \exp\Big\{\frac{C_\Omega}{\kappa}\tstep\sum_{n = 1}^{N-1}
\Big[ \|\Delta v_n\|_{L^2}\\
& \quad + \frac 1{\kappa^2} \|\nabla\Delta d_n\|_{L^2}^2 + \frac1 \kappa\big(\|\nabla\Delta d_n\|_{L^2}+ \|\nabla\Delta d_n\|_{L^2}^2 + \|\nabla v_n\|^2_{L^2}\big)\Big] \Big\} ,
 \end{aligned}
\end{equation}
where $\nu =1+  k (\|\tilde u_0\|_{L^\infty} + 2 \|\tilde u_1\|_{L^\infty} + \|\tilde u_2\|_{L^\infty})/2 $.  
 Here we used \eqref{Lipschitz_regul} and that  $\Delta d_n =0$ on $\partial\Omega$.
Thus for 
$$ \kappa \sup_{1\leq n\leq N-1} \|\Delta v_n\|^2_{L^2} + \sup_{1\leq n\leq N-1}\|\nabla \Delta d_n\|^2_{L^2} \leq R^2,
$$
and initial conditions 
$$
\|\Delta (u)_0\|^2_{L^2} + (1+ 2k  C_\Omega b)\|\nabla v_0 + (1/2)\tstep \nabla \partial_t^2 u(0)  \|^2_{L^2} \leq \eta  b^2,
$$
and appropriate $T>0$, such that  
$$
 \exp\big\{C_\Omega T (R/\kappa + (R+2R^2)/\kappa^2 + R^2 /\kappa^3)\big\}\leq \frac {\min_n \rho^{2(n-1)}} {\eta}, 
$$
 we obtain 
$$
\kappa \sup_{1\leq n\leq N-1}\|\tilde D \nabla u_n \|^2_{L^2} + \sup_{1\leq n\leq N-1}\|\Delta (u)_n\|^2_{L^2} \leq  b^2,
$$
where $0<\kappa \leq  1- 2k C_\Omega b$ and $C_\Omega$ is the constant in \eqref{H2_Delta} and \eqref{embeding_1}. This ensures
$$
\begin{aligned}
\| \{ u\}_n \|_{L^\infty} & \leq \frac 12 C_{\Omega} \big(\|(u)_n\|_{H^2} + 
\|(u)_{n-1} \|_{H^2} \big)\\
& \leq \frac 12 C_\Omega  \big(\|\Delta (u)_n \|_{L^2} + \|\Delta (u)_{n-1} \|_{L^2} \big) \leq C_\Omega  b \quad \text{ for } \; 1\leq n \leq N-1.
\end{aligned} 
$$ 
%

In the case of the corrected convolution quadrature the additional term is estimated in the same way as in \eqref{corrected_estim_v0}, with $\Delta D u_n$ and $\Delta v_0$ instead of $\nabla Du_n$ and $\nabla v_0$. 

Applying the Laplace operator  to  \eqref{main_discrete_k}  yields 
 \begin{equation}\label{Laplace_eq_k}
 \begin{aligned}
\big(1- 2k d_n\big) D^2 \Delta u_n
- 2k \Delta d_n  D^2 u_n - 4 k \nabla d_n D^2 \nabla u_n
-  \Delta^2 \{u\}_{n}\\
-a\beta\ast_{\tstep} D\Delta^2  u_n 
= 2k\big( Du_n \Delta v_n +  D\Delta u_n  v_n \big)+ 4k \nabla D u_n \nabla v_n .
 \end{aligned} 
 \end{equation}
 Considering $ \scl^{2n} D\Delta u_n$ as a test function in \eqref{Laplace_eq_k}, see Remark~\ref{rem:testing}, we obtain  
 \begin{equation}\label{estim:Laplace_k}
 \begin{aligned} 
\tstep \sum_{n=1}^{N-1} \scl^{2n} \Big[\big\langle  \big(1- 2k d_n\big) D^2 \Delta u_n, D\Delta u_n \big\rangle 
+ \big\langle \Delta \nabla \{u\}_{n}, \Delta \nabla D u_n \big\rangle 
\\
+ a \big\langle  \beta\ast_{\tstep} D \Delta \nabla   u_n,   \Delta \nabla D u_n \big\rangle \Big]
\\
=2k\tstep \sum_{n=1}^{N-1} \scl^{2n}\big\langle  \Delta d_n D^2 u_n + 2 \nabla d_n D^2 \nabla u_n, D\Delta u_n \big\rangle 
\\
+ 2k \tstep \sum_{n=1}^{N-1} \scl^{2n}\big\langle Du_n \Delta v_n + D\Delta u_n  v_n + 2\nabla D u_n \nabla v_n, D\Delta u_n \big\rangle .
 \end{aligned} 
 \end{equation} 
 For the first term in the same way as above we obtain  
 $$
 \begin{aligned} 
& \tstep \sum_{n=1}^{N-1} \scl^{2n}  \big\langle (1- 2kd_n) D^2 \Delta u_n, D\Delta u_n \big\rangle   \geq  k \tstep\int_\Omega \sum_{n = 1}^{N-2} \scl^{2(n+1)} D(\tilde u)_{n} |\tilde D\Delta u_{n}|^2dx
\\
&\qquad  +  \frac12 \int_\Omega \scl^{2(N-1)}\big(1-2k d_{N-1}\big)| \tilde D \Delta u_{N-1}|^2dx 
- \frac12 \int_{\Omega}\big(1-2k d_{1}\big)| \tilde D \Delta u_0 |^2dx.
  \end{aligned}
$$
The second and third terms in \eqref{estim:Laplace_k} are estimates in the same way as above and we have 
 $$
  \begin{aligned} 
  & \tstep \sum_{n=1}^{N-1} \scl^{2n} \Big[\big\langle \Delta \nabla \{u\}_{n}, \Delta \nabla D u_n \big\rangle 
 + \big\langle  \beta \ast_{\tstep} D \Delta \nabla D  u_n,   \Delta \nabla D u_n \big\rangle \Big]
 \\
 & \quad \geq \frac 12 \scl^{2(N-1)}\Big \| \frac{ \Delta \nabla (u_N + u_{N-1})} 2 \Big \|_{L^2}^2 - 
 \frac 12 \Big \| \frac{ \Delta \nabla (u_1 + u_{0})} 2 \Big \|_{L^2}^2
\\
&\qquad + C_\beta \tstep \sum_{n=1}^{N-1} \scl^{2n} \big\|\beta\ast_{\tstep} D \Delta \nabla   u_n \big\|_{L^2}^2.  
  \end{aligned}
 $$
 For the last term on the right-hand side of \eqref{estim:Laplace_k} we obtain
 $$
  \begin{aligned} 
  & \tstep \sum_{n=1}^{N-1} \scl^{2n} \Big | \big\langle Du_n \Delta v_n + 
  D\Delta u_n v_n +  2\nabla D u_n \nabla v_n, D\Delta u_n \big\rangle \Big| 
  \\ 
  & \qquad \leq  C_\Omega \tstep \sum_{n=1}^{N-1} \scl^{2n} \big(\|\Delta v_n\|_{L^2} + 
   \|v_n\|_{L^\infty}+ \|\nabla  v_n\|_{L^4}\big) \|\Delta Du_n\|^2_{L^2}.
  \end{aligned}
 $$
Here we used  estimates  \eqref{H2_Delta} and \eqref{embeding_1}. 
To estimate the first term on the right-hand side of \eqref{estim:Laplace_k} we first use the equation \eqref{main_discrete_k} to write 
$$
D^2 u_n = \frac{1}{(1- 2k d_n)} \Big(\Delta \{ u\}_n +   \beta\ast_{\tstep} D \Delta u_n + 2k Du_n v_n\Big)
$$
and 
$$
\begin{aligned} 
D^2 \nabla u_n = \frac{1}{(1- 2k d_n)} \Big(\Delta \nabla \{ u\}_n +   \beta\ast_{\tstep} D \Delta \nabla u_n + 2k \big(Du_n \nabla  v_n + D\nabla u_n   v_n\big)\Big) 
\\ 
+ 2k \nabla d_n \frac{1}{(1- 2k d_n)^2} \Big(\Delta \{ u\}_n +   \beta \ast_{\tstep} D \Delta u_n + 2k Du_n v_n\Big).
\end{aligned} 
$$
Then, using the Sobolev embedding theorem, yields  
$$
\begin{aligned} 
& \tstep \sum_{n=1}^{N-1} \scl^{2n} 2k \big\langle  \Delta d_n D^2 u_n, D \Delta u_n \big\rangle 
 \\
 & \leq C \tstep \sum_{n=1}^{N-1} \scl^{2n}\frac{ \|\Delta d_n\|_{L^4}}{\big\|1-2kd_n\big\|_{L^\infty} }
 \Big[\|v_n\|_{L^8} \|Du_n\|_{L^8} +  \|\Delta \{ u \}_n\|_{L^4}\Big] \| D \Delta u_n \|_{L^2} 
 \\
 & \quad + \tstep \sum_{n=1}^{N-1}\scl^{2n} \Big [C
 \Big\|\frac{ 1}{1-2kd_n}\Big\|^2_{L^\infty} 
 \|\Delta d_n\|^2_{L^4} \| D \Delta u_n \|_{L^2}^2
 + \varsigma \|\beta\ast_{\tstep} D \Delta  u_n\|^2_{L^4}\Big]
  \\
& \leq 
C \tstep \sum_{n=1}^{N-1}\scl^{2n}  \Big\|\frac{ 1}{1-2k d_n}\Big\|^2_{L^\infty}\Big(\|\Delta d_n\|_{L^4}^2 +  \|v_n \|_{L^8}^2 \Big)
  \| D \Delta u_n \|_{L^2}^2 \\
 & \quad + C\tstep \sum_{n=1}^{N-1}\scl^{2n}  \| \nabla \Delta  \{ u \}_n \|^2_{L^2}
  + \varsigma \tstep \sum_{n=1}^{N-1} \scl^{2n} \|\beta \ast_{\tstep} D \nabla \Delta   u_n\|^2_{L^2}
\end{aligned} 
$$
and 
$$
\begin{aligned} 
&\tstep \sum_{n=1}^{N-1} \scl^{2n}  \big|\langle  \nabla d_n D^2 \nabla u_n, D \Delta u_n \rangle\big|
\leq C \tstep \sum_{n=1}^{N-1}\scl^{2n}\left[ \frac{\|\nabla d_n\|_{L^\infty}}{\big\|1-2k d_n\big\|_{L^\infty}}\Big[ \|\Delta \nabla \{ u\}_n\|_{L^2} 
\right.\\
& \quad +  \|D u_n\|_{L^\infty}
\|\nabla v_n\|_{L^2} 
+  \|D \nabla u_n\|_{L^4}
\|v_n\|_{L^4} +  \| \beta\ast_{\tstep} D \Delta \nabla  u_n\|_{L^2}\Big] \\
&\left.  + \frac{ \|\nabla d_n\|^2_{L^6}}{\big\|1-2k d_n\big\|^2_{L^\infty} } 
 \Big[\|\Delta \{ u\}_n\|_{L^6} +   \|\beta\ast_{\tstep} D \Delta u_n\|_{L^6} +  \|Du_n\|_{L^\infty} \|v_n\|_{L^6}\Big] \right] \hspace{-0.1 cm}
 \|D \Delta u_n \|_{L^2} 
\\
& \leq C_\Omega \tstep \sum_{n=1}^{N-1}\scl^{2n} \Big\|\frac{ 1}{1-2k d_n}\Big\|^2_{L^\infty} \Big(\Big\|\frac{ 1}{1-2k d_n}\Big\|^2_{L^\infty} \|\nabla d_n\|_{L^6}^4 + \|\nabla d_n\|_{L^6}^4 +  \|\nabla d_n\|_{L^\infty}^2\\
& \quad + \|\nabla v_n\|_{L^2}^2 \Big) 
\|D \Delta u_n \|_{L^2}^2  +   \tstep \sum_{n=1}^{N-1} \scl^{2n} \Big( \|\nabla \Delta \{ u\}_n\|^2_{L^2}  +
\varsigma \| \beta \ast_{\tstep} D \nabla \Delta   u_n\|^2_{L^2} \Big).
\end{aligned} 
$$
Here we used \eqref{H2_Delta}, \eqref{Lipschitz_regul},  and
$$
\|\Delta w\|_{L^p} \leq C_\Omega \| \Delta  w\|_{H^1}\leq  C_\Omega \|\nabla \Delta  w\|_{L^2}, \; \text{ for }  1< p \leq 6,
$$
where $\Delta w =0$ on $\partial\Omega$.
Choosing $\varsigma>0$  sufficiently small  yields 
$$
\begin{aligned} 
 &\scl^{2(N-1)} \kappa \|\tilde D \Delta u_{N-1} \|_{L^2}^2 +   \scl^{2(N-1)} \|\Delta \nabla  (u)_{N-1}\|^2_{L^2} 
 \\
 &\quad  +
  C_\beta\tstep \sum_{n=1}^{N-1} \scl^{2n}  
\| \beta\ast_{\tstep} D \Delta \nabla  u_n\|^2_{L^2}
 \leq  \nu  \|\tilde D \Delta u_0 \|_{L^2}^2 + \| \Delta \nabla (u)_0 \|_{L^2}^2
 \\
 & \qquad + C_\Omega \tstep \sum_{n=1}^{N-1} \scl^{2n}
\Big(\|\Delta \nabla \{u\}_n \|_{L^2}^2  + \Big\|\frac{ 1}{1-2kd_n}\Big\|^2_{L^\infty}  \Big[\|\Delta v_n \|_{L^2}^2 + \|\nabla \Delta d_n \|_{L^2}^2  \\
& \qquad \qquad + \Big(1+ \Big\|\frac{ 1}{1-2kd_n}\Big\|^2_{L^\infty}\Big) \|\Delta d_n \|_{L^2}^4 \Big] \|D \Delta u_n \|_{L^2}^2 \Big).
\end{aligned}
$$
Considering $T$ such that 
$$
\begin{aligned}
 \Big[\nu\|\tilde D \Delta u_0 \|_{L^2}^2 + \| \Delta \nabla (u)_0  \|_{L^2}^2 \Big]\exp\Big\{C_\Omega T \max\Big\{1, \frac{R^2 + R^4}{\kappa^3} + \frac{R^4}{\kappa^5}\Big\}\Big\} \\
 \leq \min_n \scl^{2(n-1)} R^2,
 \end{aligned}
$$
and using the discrete Gr\"onwall inequality,  we obtain 
$$
\scl^{2n}\kappa \|\tilde D  \Delta u_n \|_{L^2}^2 +   \scl^{2n} \|\Delta \nabla  (u)_n\|_{L^2}^2 \leq \scl^{2n}R^2 \quad \text{ for all } \; n=1, \ldots, N-1,
$$ 
and then also 
$$
 \tstep \sum_{n=1}^{N-1} \scl^{2n}  
\| \beta\ast_{\tstep} D \Delta \nabla  u_n\|^2_{L^2} \leq C.
$$

In the case of the corrected convolution quadrature the additional term is estimated in the same way as in \eqref{corrected_estim_v0}, with $\nabla \Delta D u_n$ and $\nabla \Delta v_0$ instead of $\nabla Du_n$ and $\nabla v_0$.

As next we have to show the contraction property for the map $\tilde u_n \to u_n$. 
Considering the equation for $w_n = u^{1}_n -  u^{2}_n$, where $u^j_n$ satisfies \eqref{main_discrete_k} for $\tilde u_n^j$, with $j=1,2$, taking $-\scl^{2n} \Delta  Dw_n$ as a test function and summing over $n$, yields
\begin{equation} 
 \begin{aligned} 
 & \tstep \sum_{n=1}^{N-1} \scl^{2n} \Big[\big\langle  (1- 2k  d_n^1) D^2 \nabla w_n, D\nabla w_n \big\rangle + \big\langle \Delta  \{w\}_{n}, \Delta  D w_n \big\rangle 
\\
&\qquad  +  \big\langle a \beta \ast_{\tstep} D \Delta    w_n,   D \Delta  w_n \big\rangle \Big]
\\
& =2k\tstep \sum_{n=1}^{N-1}\Big[ \scl^{2n}\big\langle D w_n  \nabla  v_n^1 
+ D\nabla w_n  v_n^1 +  \nabla d_n^1 D^2 w_n, D\nabla w_n \big\rangle  \\
&\qquad  +  \scl^{2n}\big\langle (d^1_n - d^2_n)
D^2 \nabla u^2_n + \nabla (d^1_n - d^2_n)
D^2 u^2_n, D\nabla w_n \big\rangle
\\
& \qquad + \scl^{2n}\big\langle (v_n^1- v_n^2)D \nabla u^2_n + \nabla(v_n^1- v_n^2)D  u^2_n, D\nabla w_n \big\rangle \Big],
 \end{aligned}
\end{equation}
where $d_n^j = \{\tilde u^j\}_n$   and $v_n^j = D\tilde u^j_n$ for $j=1,2$.
Performing estimates similar as above we obtain 
\begin{equation*} 
 \begin{aligned} 
&\rho^{2(N-1)} \kappa\| \nabla \tilde D(u^1_{N-1} - u^2_{N-1})\|^2_{L^2} + \rho^{2(N-1)}\| \Delta \big(u^1 - u^2\big)_{N-1}\|^2_{L^2} 
\\
& \quad + (2C_\beta- \varsigma)\tstep \sum_{n=1}^{N-1} \scl^{2n} \|\beta \ast_{\tstep} D \Delta \left( u_n^1 - u_n^2\right) \|_{L^2}^2
\\
&\leq C_\Omega\tstep \sum_{n=1}^{N-1} \scl^{2n} \big[ \|D \Delta u_n^2\|^2_{L^2} +\|D^2 \nabla u_n^2\|_{L^2}\big] \Big[ \|\Delta (d^1_n - d^2_n) \|_{L^2}^2
+ \|\nabla(v_n^1- v_n^2)\|_{L^2}^2 \Big] \\
&\quad +
 C_\Omega\tstep \sum_{n=1}^{N-1} \scl^{2n} \Big(\| \Delta \big(u^1 - u^2\big)_n\|^2_{L^2} + \Big[1+ \|\Delta v_n^1\|_{L^2}+ \frac 1 { \kappa} \|\nabla d_n^1\|_{L^\infty}\|\nabla v_n^1\|_{L^2} \\
 & \qquad \qquad \quad + \frac 1 {\kappa^2} \|\nabla d_n^1\|^2_{L^\infty} \big(1+ \|D^2 \nabla u_n^2\|_{L^2}\big) + \|D^2 \nabla u_n^2\|_{L^2} \Big]\|\nabla \tilde D (u^1_n - u^2_n)\|^2_{L^2}\Big)  \\
 &\quad  + C\tstep \sum_{n=1}^{N-2} \scl^{2(n+1)} \|D(\tilde u^1)_n \|_{L^\infty}\|\nabla \tilde D (u^1_n - u^2_n)\|^2_{L^2}.
 \end{aligned}
\end{equation*}
Applying the discrete  Gr\"onwall inequality  we obtain the contraction property for an appropriate $\tau = N \tstep$. Iterating over the time interval we obtain that there exists a fixed point of the map given by equation~\eqref{main_discrete_k} and the corresponding stability conditions. 
\end{proof}

\section{Error estimate} 
As next we derive the error estimates for the time-discretization scheme \eqref{main_discrete_k}. 

\subsection{Expected smoothness at $t = 0$}\label{sec:reg} 

Consider first the linear equation
\begin{equation}\label{eq:lin1}
\frac1{c^2}\partial_t^2 u -\Delta u -a\beta\ast  \Delta \partial_t u = f,
\end{equation}
with smooth $f$.  The choice of  $\beta$ (either $\betaA$ or $\betaB$) implies that for $k \in \mathbb{N}_0$
\[
\int_0^t \beta(t-\tau) \tau^k d\tau \sim \frac1{\Gamma(\mu)}\int_0^t (t-\tau)^{\mu-1}\tau^kd\tau = \frac{\Gamma(k+1)}{\Gamma(k+1+\mu)} t^{k+\mu} \qquad \text{as } t \to 0^+.
\]
Considering similar arguments as in \cite[Remark~2.10]{Baker_Banjai}, we thus expect that the behaviour at $t =0$ of the solution to the linear problem to be given by 
{
\[
  \partial_t^2u (t) = c^2\left(f +\Delta u_0+\frac{1}{\Gamma(1+\mu)} t^{\mu} \Delta v_0 +\o(t^\mu)\right).
\]
}

In the nonlinear case, as long as no singularity develops and   $2ku < 1$ continues  to hold, 
we expect the singularity at $t = 0$ to be of the same type
\[
  \partial_t^2u (t) \sim w_0+z_0 t^{\mu} + \o(t^\mu), 
\]
where $\mu \in (0,1)$. This motivates the following assumption on the smoothness of the solution that will allow us to develop realistic error estimates.

\begin{assumption} \label{ass:reg_ass}
Assume that 
\[
u \in C^2([0,T]; H^2(\Omega)) \cap  C^3((0,T]; H^2(\Omega)) \cap  C^4((0,T]; L^2(\Omega))
\]
and that there exist constants $c_k \geq 0$, for $k = 1,\dots, 4$, such that
 \begin{equation*}
 \begin{aligned}
   & \|\partial_t^{k}u\|_{H^2(\Omega)} \leq C (1+c_k t^{2+\mu-k}) \qquad && \text{ for  } \; t \in (0, T] \; \text{ and } \; k=1,2,3,  \\
  & \|\partial_t^{4}u\|_{L^2(\Omega)} \leq C (1+c_4 t^{\mu-2})
   \qquad && \text{ for  } \; t \in (0, T].
   \end{aligned}
 \end{equation*}
\end{assumption}

In what follows, for $u \in C[0,T]$ we shall use  the following  notation 
$$
\begin{aligned} 
& Du(t_n) = \frac{ u(t_{n+1}) - u(t_{n-1})}{2 \tstep}, \quad && \tilde D u(t_n) = \frac{u(t_{n+1}) - u(t_n)}{\tstep}, \\
&D^2 u(t_n) = \frac{ u(t_{n+1})
-2 u(t_n) + u(t_{n-1})}{ \tstep^2}, && \text{ for } \;\;  t_n \in [\tstep,T-\tstep],  \\
&\{u\}(t_n) = \frac14\left( u(t_{n+1})
+2 u(t_n) + u(t_{n-1})\right), && \text{ for } \;\;  t_n \in [\tstep,T-\tstep].
\end{aligned} 
$$

We will require the following lemma proved in \cite{Baker_Banjai} for $\beta = \betaA$ and $r = 0$. 
\begin{lemma}\label{lemma_error_beta}
For $u \in C^3(0,T]$ and any $t_n \in [\tstep,T]$
  \begin{enumerate}[(a)]
  \item  
\[
  \begin{split}      
\left|\beta \ast \partial_t u(t_n)-\beta \ast_\tstep D u(t_n)\right| &\leq C \bigg\{ t_n^{\mu-1} \partial_tu(\tstep)\tstep\\
&+\left. \tstep^2\left(\partial_t^2u(\tstep) t_n^{\mu-1}+\int_{\tstep}^{t_n}(t_n-\tau)^{\mu-1}|\partial_t^3 u(\tau)|d\tau\right)  \right\}.
\end{split}
\]  
\item 
\[
\left|\beta \ast \partial_t u(t_n)-\beta \tilde\ast_\tstep D u(t_n)\right| \leq C
\tstep^2\left(\partial_t^2u(\tstep) t_n^{\mu-1}+\int_{\tstep}^{t_n}(t_n-\tau)^{\mu-1}|\partial_t^3 u(\tau)|d\tau\right).
\]  
  \end{enumerate}
\end{lemma}
\begin{proof}
  The result for $\beta =\betaA$ and $r = 0$, i.e., $\beta = \frac1{\Gamma(\mu)} t^{\mu-1}$, is shown in \cite[Lemma~4.4]{Baker_Banjai} and \cite[Lemma~4.5]{Baker_Banjai} respectively. Looking closely at the proof as well as   \cite[Lemma~4.2]{Baker_Banjai} and \cite[Theorem~2.1]{Lubrev}, it can be seen that the only property of the kernel used is $|\hat\beta(z)|\leq C |z|^{-\mu}$, $z \notin \mathbb{C} \setminus (-\infty,0]$, which holds for both $\betaA$ and $\betaB$; see \eqref{eq:hatbeta}.
\end{proof}

\begin{theorem}\label{thm:conv}
Let $u$ be the solution of \eqref{main_discrete} satisfying  Assumption~\ref{ass:reg_ass} and with initial data satisfying the conditions of Lemma~\ref{lem:disc_bounds}. Let $u_n$ be the solution of the semi-discrete scheme. If the uncorrected CQ is used, the error $w_n = u_n - u(t_n)$  satisfies the estimate
 \begin{equation}
   \sup\limits_{1\leq n \leq N-1}\| \tilde Dw_n\|_{L^2(\Omega)}
   + \sup\limits_{1\leq n \leq N-1} \| \nabla ( w)_n\|_{L^2(\Omega)} = \O(\tstep).
 \end{equation}
If the corrected CQ is used the error bound becomes
 \begin{equation}
   \sup\limits_{1\leq n \leq N-1}\| \tilde Dw_n\|_{L^2(\Omega)}
   + \sup\limits_{1\leq n \leq N-1} \| \nabla ( w)_n\|_{L^2(\Omega)} = \O(\tstep^{1+\mu}).
 \end{equation}
\end{theorem}
\begin{proof} 
Consider the difference between the solution and approximation denoted by $w_n = u_n - u(t_n)$  to obtain
\begin{equation} \label{error_1}
\begin{aligned}
  \big(1-2k\{u\}_n\big)D^2 w_n
 -\Delta \{w_n\}_{n}    -\beta \ast_{\tstep}\Delta Dw_n  = 2k D w_n\big(Du_n + D u(t_n)\big) \\
 \quad + 2k\{w\}_nD^2u(t_n)+
 \varepsilon_n  + \sigma_n +  \delta_n + \kappa_n + \theta_n,
 \end{aligned} 
\end{equation}
where 
$$
\begin{aligned} 
\varepsilon_n &= (1- 2k u(t_n))\big( D^2  u(t_n) -\partial_t^2 u(t_n) \big)\\
\sigma_n &= 2k\big(\{u\}(t_n) -  u(t_n)\big) D^2 u(t_n), \\
\delta_n &= \beta\ast_\tstep \Delta D u(t_n) - \beta \ast \Delta \partial_t u(t_n),\\
\kappa_n & = \Delta \big(\{u\}(t_n) - u(t_n)\big), \\
\theta_n & = 2k [(Du(t_n))^2 - (\partial_t u(t_n))^2]. 
\end{aligned} 
$$
Considering $Dw_n$ as a test function in a weak formulation of \eqref{error_1} and summing  over $n=1, \ldots, N-1$, yield 
\begin{equation} \label{error_2}
\begin{aligned}
 & \tstep \sum_{n=1}^{N-1}\Big[\big\langle \big(1-2k\{u\}_n\big)D^2 w_n, Dw_n \big\rangle
 +\big \langle \nabla \{w_n\}_{n}, \nabla D w_n \big\rangle  \Big]\\
 & + 
 \tstep \hspace{-0.05 cm } \sum_{n=1}^{N-1}\hspace{-0.05 cm }  a \big\langle \beta \ast_{\tstep} \nabla Dw_n, \nabla Dw_n \big\rangle \hspace{-0.05 cm }
= 2k \tstep \hspace{-0.05 cm } \sum_{n=1}^{N-1} \hspace{-0.05 cm } \big\langle  D w_n\big[Du_n + D u(t_n)\big], Dw_n \big\rangle  \\
& +2k \tstep\sum_{n=1}^{N-1} \hspace{-0.05 cm } \big\langle  \{w\}_nD^2u(t_n), Dw_n \big\rangle + 
\tstep \hspace{-0.05 cm } \sum_{n=1}^{N-1}\hspace{-0.05 cm } \big\langle\varepsilon_n  + \sigma_n +  \delta_n + \kappa_n + \theta_n, Dw_n \big\rangle.
 \end{aligned} 
\end{equation}
Using the boundedness of $Du_n$ and $Du(t_n)$ we obtain 
$$
\big\langle  D w_n\big(Du_n + D u(t_n)\big), Dw_n \big\rangle \leq C \big(\|Du_n\|_{L^\infty} + \|\partial_t u\|_{L^\infty}\big) \|Dw_n\|^2_{L^2}. 
$$
Similarly using that $D^2u(t_n) \in L^4(\Omega)$, together with the Sobolev embedding inequality,  we have
\begin{equation*}
\begin{aligned} 
  \big\langle  \{w\}_nD^2u(t_n), Dw_n \big\rangle & \leq C\|D^2u(t_n)\|_{L^4}\left(\|\{w\}_n\|_{L^4}^2+\|Dw_n\|_{L^2}^2\right)\\
 & \leq C \|\partial_t^2 u\|_{L^4}\left(\|\nabla (w)_n\|_{L^2}^2+ \|\nabla (w)_{n-1}\|_{L^2}^2+\|Dw_n\|_{L^2}^2\right) .
  \end{aligned} 
\end{equation*}
The last term in \eqref{error_2} we can write as 
$$
\begin{aligned} 
 \tstep \hspace{-0.05 cm } \sum_{n=1}^{N-1} \hspace{-0.05 cm } \Big|\big\langle\varepsilon_n  + \sigma_n +  \delta_n + \kappa_n + \theta_n, Dw_n \big\rangle\Big| \leq 
 C \Big[ \Big(\tstep \hspace{-0.05 cm } \sum_{n=1}^{N-1} \hspace{-0.1 cm }\|\varepsilon_n\|_{L^2} \Big)^2 
 + \Big(\tstep \hspace{-0.05 cm } \sum_{n=1}^{N-1} \hspace{-0.1 cm } \|\sigma_n\|_{L^2} \Big)^2  \\
 + 
 \Big(\tstep \hspace{-0.05 cm }\sum_{n=1}^{N-1}\hspace{-0.1 cm }\|\delta_n\|_{L^2} \Big)^2  + \Big(\tstep \hspace{-0.05 cm } \sum_{n=1}^{N-1}\hspace{-0.1 cm } \|\kappa_n\|_{L^2} \Big)^2 
 + \Big(\tstep \hspace{-0.05 cm } \sum_{n=1}^{N-1}\hspace{-0.1 cm } \|\theta_n\|_{L^2}\Big)^2 \Big]  + \varsigma \hspace{-0.1 cm } \sup_{1 \leq n \leq N-1} \hspace{-0.15 cm } \|Dw_n\|_{L^2}^2.
\end{aligned}
$$
Assumption~\ref{ass:reg_ass} allows us to bound the perturbation terms as in \cite{Baker_Banjai} by
\begin{equation*} 
 \begin{aligned}
 \tstep\sum_{n=1}^{N-1}  \|\varepsilon_n \|_{L^2} & = \tstep\sum_{n=1}^{N-1} \big\|  (1- 2k u(t_n))\big[ D^2  u(t_n) -\partial_t^2 u(t_n) \big] \big\|_{L^2} \\
 \leq
C  \tstep &\int_0^{2\tstep} \hspace{-0.2 cm } \| \partial_t^3 u \|_{L^2} dt  + C  \tstep^2 \int_\tstep^{T} \|\partial_t^4 u\|_{L^2} dt  \leq C \tstep^2(1+ (c_3+c_4) \tstep^{\mu-1}),  \\ 
\tstep\sum_{n=1}^{N-1}  \|\sigma_n \|_{L^2} & =  \tstep\sum_{n=1}^{N-1}  \big\|\big[\{u\}(t_n) -  u(t_n)\big] D^2 u(t_n)\big\|_{L^2} = \tstep\sum_{n=1}^{N-1} \frac{\tstep^2}{4} \|D^2 u(t_n)\|^2_{L^4} \\
&  \leq C \tstep^2  \|\partial_t^2 u\|^2_{L^2(0,T; L^4(\Omega))}  \leq C \tstep^2,
 \end{aligned}
\end{equation*} 
where we used that $\partial_t^2 u \in L^2(0,T;H^1(\Omega))$.
Using the assumption that $\|\Delta \partial_t^2 u \|_{L^1(0,T; L^2(\Omega))}$ is bounded yields
\begin{equation*} 
\begin{aligned}
 \tstep\sum_{n=1}^{N-1}  \| \kappa_n\|_{L^2} &=\tstep \sum_{n=1}^{N-1} \|\Delta (\{u\}(t_n) - u(t_n)) \|_{L^2} \\
 &\leq
 C \tstep^2  \|\Delta \partial_t^2 u \|_{L^1(0,T; L^2(\Omega))}
 \leq C \tstep^2.
 \end{aligned}
\end{equation*}
Notice that estimate for $\|\Delta \partial_t^2 u \|_{L^2((0,T)\times\Omega)}$  can be shown for initial data $u_0 \in H^4(\Omega)$, $v_0 \in H^3(\Omega)$ in the similar way as the estimate for $\|\nabla \partial_t^2 u \|_{L^2((0,T)\times \Omega)}$ in Lemma~\ref{lem_apriori}.

To bound the  CQ approximation error we use  Lemma~\ref{lemma_error_beta} and \cite[Lemma~4.1]{Baker_Banjai} to conclude
\begin{equation}
 \tstep\sum_{n=1}^{N-1}  \| \delta_n\|_{L^2} = \tstep \sum_{n=1}^{N-1} \big\|\Delta \big(\beta \ast_\tstep  D u(t_n) - \beta \ast \partial_t u(t_n)\big) \big\|_{L^2} \leq C\tstep. 
\end{equation}
Instead if we are using the correction, then
\begin{equation}
 \tstep\sum_{n=1}^{N-1}  \| \delta_n\|_{L^2} = \tstep \sum_{n=1}^{N-1} \big\|\Delta  \big(\beta \tilde\ast_\tstep  D u(t_n) - \beta \ast \partial_t u(t_n)\big) \big\|_{L^2} \leq  C\tstep^2.
\end{equation}
The last error term is estimated as 
\begin{equation} \label{estim_remind_3} 
\begin{aligned} 
 & \tstep\sum_{n=1}^{N-1} \|\theta_n\|_{L^2} = 
2k \tstep\sum_{n=1}^{N-1} 
 \|(Du(t_n))^2 - (\partial_t u(t_n))^2\|_{L^2}
 \\
 & \leq C \tstep\sum_{n=1}^{N-1} \|Du(t_n)- \partial_t u(t_n)\|_{L^2} 
 \big(\|\partial_t u(t_n)\|_{L^\infty} + \|Du(t_n)\|_{L^\infty}\big) 
 \\
 & \leq C \tstep\sum_{n=1}^{N-1} \tstep^2 \left(1+c_3\int_{0}^T t^{\mu-1} dt\right) \leq C \tstep^2.
\end{aligned}
\end{equation}
 Hence combining all estimates and applying Gr\"onwall inequality  yields 
 \begin{equation} \label{estim_comb} 
\begin{aligned} 
&\kappa \sup_{1\leq n \leq N-1}\|\tilde Dw_{n}\|^2_{L^2} + \sup_{1\leq n \leq N-1}  \big\|\nabla (w)_{n} \big\|^2_{L^2}
\leq
 \Big(\big\|\nabla (w)_0 \big\|_{L^2}^2 \\
 & +  (1+2k C_\Omega b) \|\tilde Dw_{0}\|^2_{L^2}
+ C\tstep^4\big[1+ 
(c_3+ c_4) \tstep^{2(\mu-1)}\big]+ C_1\tstep^2 \Big)\times 
\\
&\times \exp\big\{C \big(\|\partial_t u \|_{L^\infty}+\|\partial_t^2 u \|_{L^4} + \sup_{1\leq n \leq N-1}\|Du_n\|_{L^\infty}\big)\big\},
\end{aligned}
\end{equation}
where $0<\kappa \leq 1 - 2k\|u\|_{L^\infty}$, the constant $C$ may depend on the final time $T$,  and $C_1 =0$ when using  the corrected CQ scheme.

As $w_0 = 0$, it remains to estimate $\|\nabla w_1\|$ and $\|\tilde Dw_0\|_{L^2} = \tstep^{-1}\|u_1-u(t_1)\|_{L^2}$. The choice of $u_1$ in \eqref{eq:initial}, Taylor expansion and Assumption~\ref{ass:reg_ass} ensure that the two terms are of size  $\O(\tstep^{2+\mu})$ and $\O(\tstep^{1+\mu})$ respectively.
\end{proof}

\begin{remark}
 Notice that $Dw_0=0$, thus we do not have any additional contribution for corrected CQ.
\end{remark}

\begin{remark}
  In Theorem~\ref{thm:conv} we obtained the estimate in the natural discrete energy norm. If instead of testing by $Dw_n$, we test with $\tstep \sum_{j = n}^{N-1} (w)_j$ we can  obtain an estimate on 
\[
\|w_N\|^2_{L^2} + \left\|\tstep \sum_{n = 1}^{N-1} \nabla (w)_n\right\|^2_{L^2}.
\]
In doing this, lower regularity assumption in space of the solution could be made in Assumption~\ref{ass:reg_ass}, namely $H^1(\Omega)$ instead of $H^2(\Omega)$.
\end{remark}

\section{Numerical Experiments}
\label{sec:numerics}

Coupling the time discretization \eqref{main_discrete} with the piecewise linear Galerkin finite element space discretization  we obtain a fully discrete scheme. Denoting by $V^h\subset H^1_0(\Omega)$ the space of piecewise linear finite element functions we have that the fully discrete solution $u^h_n \in V^h$ satisfies
\begin{equation}
    \label{equ:full_disc}
    \begin{split}      
    \big\langle \big( 1-2 k \{u^h\}_n \big) D^2 u_n^h ,v  \big\rangle + \big\langle \nabla\{  u^h \}_n,\nabla v \big\rangle &+ a \big\langle \beta \ast_{\tstep} D \nabla  u_n^h , \nabla v \big\rangle\\ &= 2k \big\langle \big( D u_n^h \big)^2 , v \big\rangle ,
    \end{split}
\end{equation}
for all $v\in V^h$, $n = 1,\dots$. The initial data is set to
\[
u_0^h = P_hu_0, \qquad u_1^h = u_0^h+\tstep P_hv_0 + \frac12 \tstep^2 P_h \partial_t^2u(0),
\]
where $P_h \colon L^2(\Omega) \rightarrow V_h$ is the $L^2$ orthogonal projection and $P_h\partial_t^2u(0)$ is obtained from \eqref{eq:dttu0}. 
Throughout the numerical experiments we set $\beta = \betaA$.

\subsection{1D Experiments}
\label{subsec:1Dnum}

We first report on a series of experiments in 1D. To solve \eqref{equ:full_disc} at each time step for  $u_{n+1}$ we  use a Newton iteration as described in~\cite[Chapter 7.1.2]{Baker2022}. In space we use a uniform mesh with spatial meshwidth $h > 0$.

\subsubsection{Test 1: Convergence}
In this numerical experiment we solve \eqref{eq:1} on  $\Omega = (-1,1)$ with initial data given by
\[
u_0 = \sin(\pi x) \quad \text{ and } \quad v_0 = \sin(\pi x),
\]
 choose the parameters $k = 0.09$, $r=0$ and $a = 30$ and set the final time to $T = 1/2$. As error measure we use the maximum over time of the discrete energy error: 
\begin{equation}
    \label{equ:errorMORETIME}
    \begin{split}      
     \text{error} =& \max_{1\leq n \leq N} \left\| \frac{\uex_n - \uex_{n-1}}{\tstep} - \frac{u_n^h - u_{n-1}^h}{\tstep} \right\|_{L^2}\\
&+ \max_{1\leq n \leq N}  \left\|\nabla \frac{\uex_n + \uex_{n-1}}{2} - \nabla\frac{u_n^h + u_{n-1}^h}{2} \right\|_{L^2}.
    \end{split}
\end{equation}
As the exact solution is not available, for $\uex$ we use a numerical solution on a fine mesh with spatial mesh-width $h = 1.7 \times 10^{-3}$ and $\tstep^{\text{ex}} = 3.1 \times 10^{-4}$.  The same spatial-mesh is used for $u_n^h$ with a range of time-steps $\tstep$. Theorem~\ref{thm:conv} predicts  $\O(\tstep)$ convergence of the error if no correction is used and  $\O(\tstep^{1+\mu})$ for the scheme with the correction.

The numerical results for the version without the correction is shown in Figure~\ref{fig:conv_nocorr}. We see that similar rate of convergence is seen for both values of $\mu$ shown and that it is close to the predicted linear convergence. In Figure~\ref{fig:conv_corr}, where the convergence of the corrected scheme is shown, we see  better convergence for larger $\mu$ closely following the predicted order $\O(\tstep^{1+\mu})$.

\begin{figure}
    \centering
    \includegraphics[scale = 0.6]{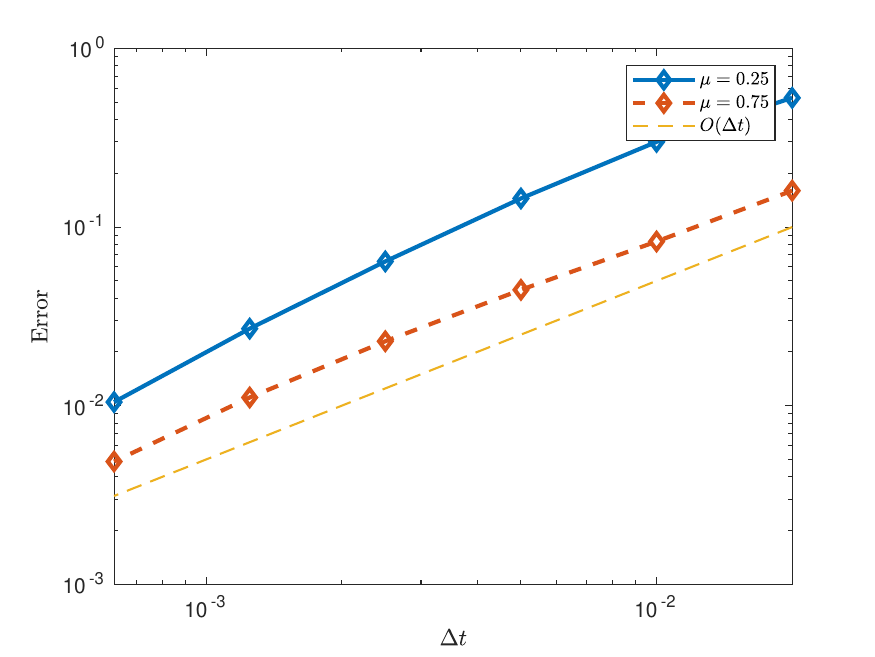}
    \caption{\label{fig:conv_nocorr} Convergence of the maximum energy error for the numerical scheme without the correction term for $\mu = 0.25$ and $\mu = 0.75$. Predicted convergence order of $\O(\tstep)$ is also shown.}
\end{figure}

\begin{figure}
    \centering
    \includegraphics[scale = 0.6]{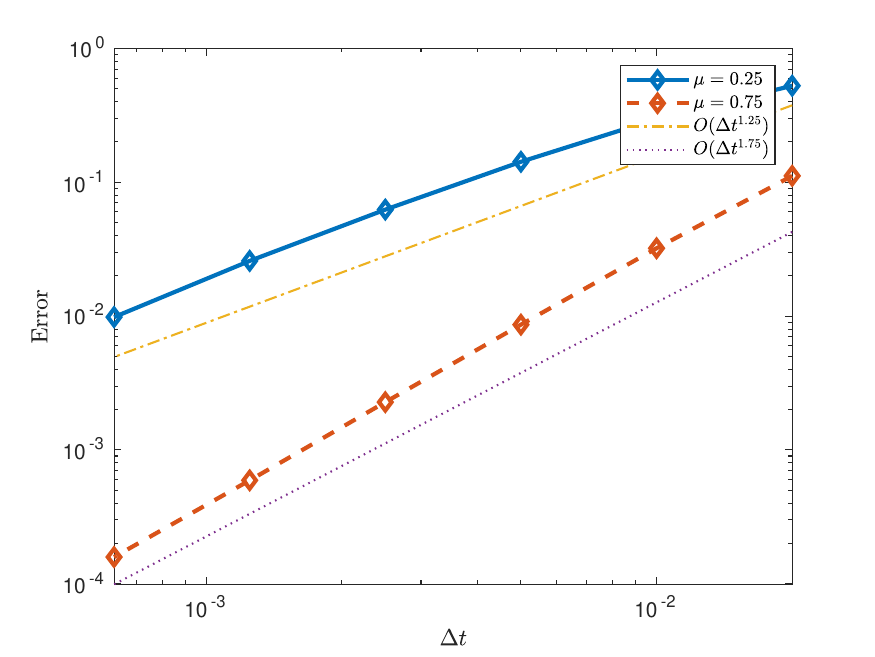}
    \caption{\label{fig:conv_corr} Convergence of the maximum energy error for the numerical scheme with the correction term included for $\mu = 0.25$ and $\mu = 0.75$. Predicted convergence order of $\O(\tstep^{1+\mu})$ is also shown.}
\end{figure}

\subsubsection{Test 2: Changing $k$}

For the remaining 1D experiments we solve \eqref{eq:1} on $\Omega = (0,20)$  and, unless otherwise stated, consider the example with initial data given by a Gaussian
\begin{equation}
    \label{equ:inidata}
    u_0 = 5e^{-\frac{(x-10)^2}{2}} \quad  \text{ and } \quad v_0 = 0.
\end{equation}
Note that while strictly speaking $u_0 \notin H_0^1(\Omega)$,  $u_0$ is zero close to machine precision on  the boundary of $\Omega$.


In Figure~\ref{fig:kchangenoFD} we show the solution of \eqref{eq:1} without the fractional derivative ($a=0$) at time $T=4$ for various choices of $k$. This figure shows that as $k$ gets larger and $(1-2ku)\to 0$ the nonlinearity has a stronger effect on the wave form, resulting in the formation of a sawtooth shape.

\begin{figure}[H]
    \centering
    \includegraphics[scale=0.35]{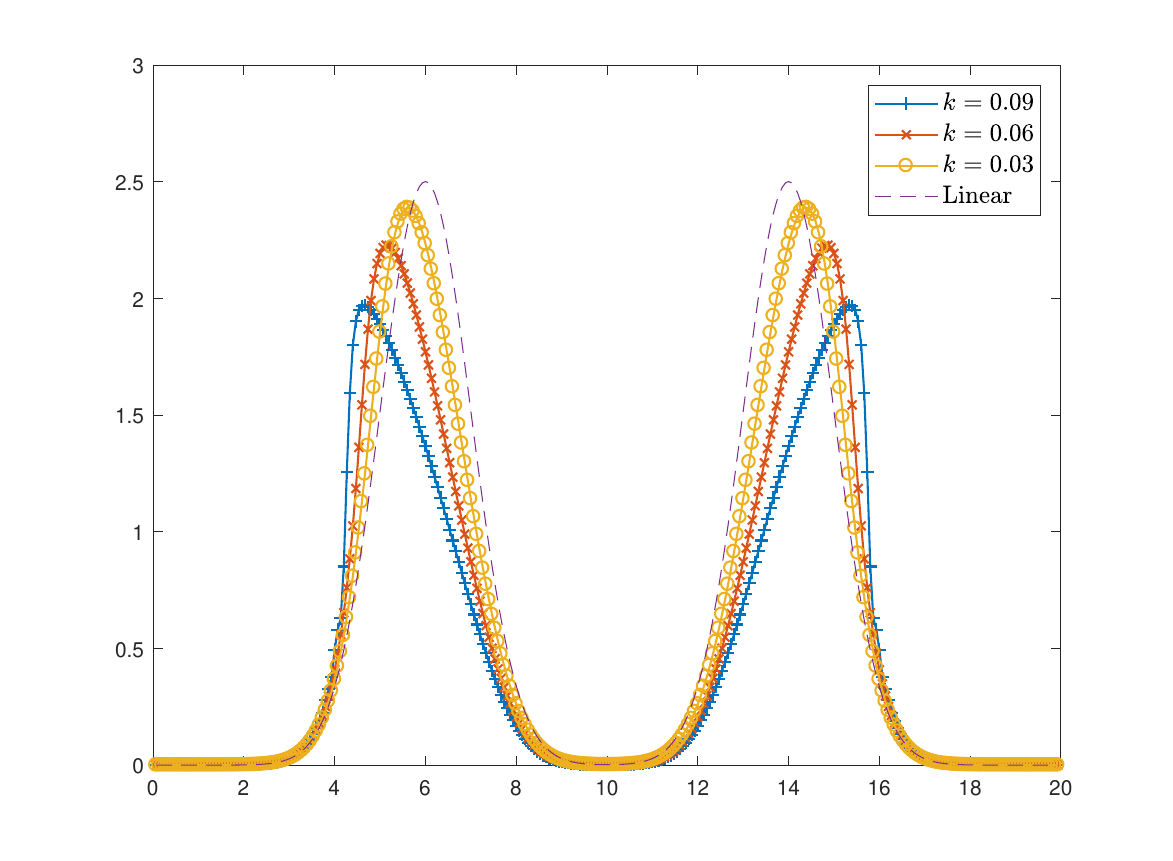}
    \caption{\label{fig:kchangenoFD} Solution of \eqref{eq:1} at $T=4$ approximated with the scheme \eqref{equ:full_disc} with $a=0$ for various values of $k$.}
\end{figure}

When we reincorporate the fractional derivative, choosing $a=1$, $r=0$ and $\mu =0.5$, we still see the damping from the nonlinearity but no longer observe the sawtooth formation. Instead the strong fractional damping term controls the form of the solution, causing more parabolic-like behaviour, i.e, the solution is trying to disperse rather than form a travelling wave; see Figure~\ref{fig:kchangeFD}. 
\begin{figure}[H]
    \centering
    \includegraphics[scale=0.4]{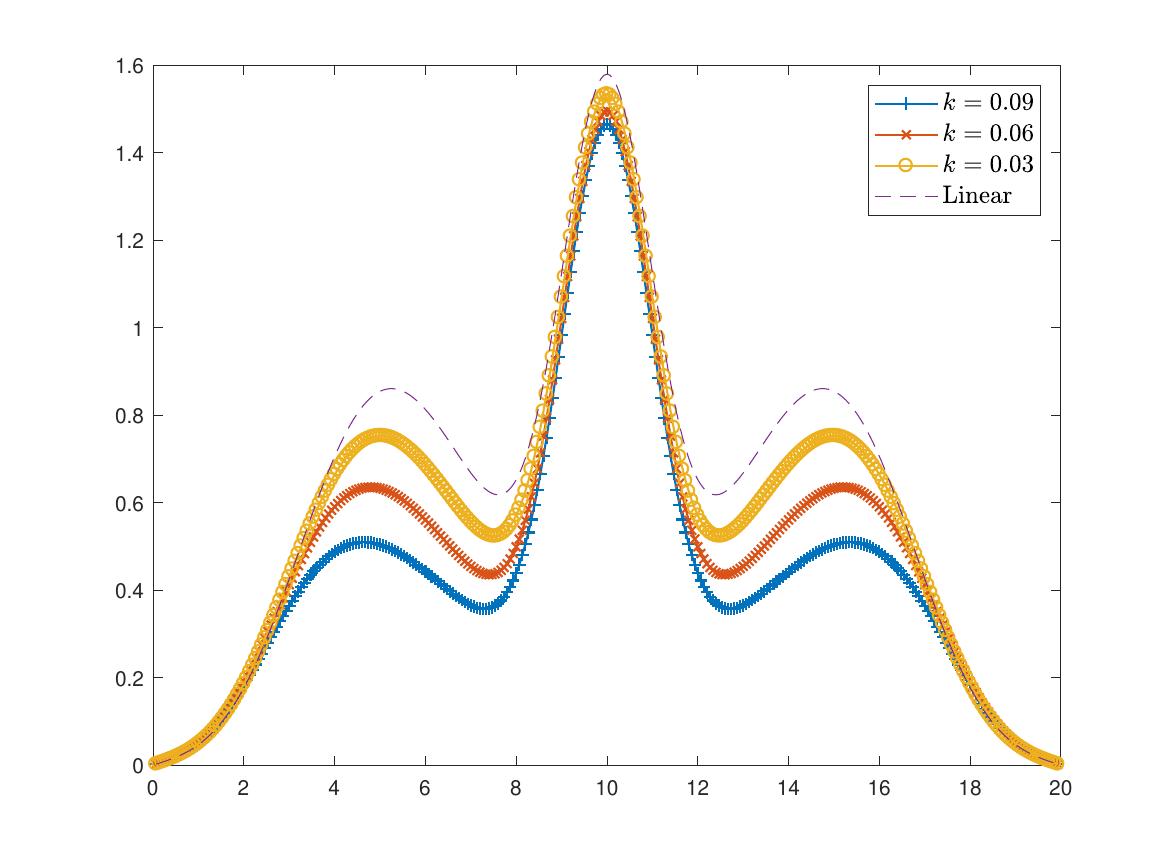}
    \caption{\label{fig:kchangeFD} Solution of \eqref{eq:1} at $T=4$ approximated using the scheme \eqref{equ:full_disc} with $a=1$, $\beta =\betaA$, $\mu = 0.5$, and $r=0$, for various values of $k$.}
\end{figure}

\subsubsection{Test 3: Changing  $a$}

This test investigates how changing the size of the coefficient scaling the fractional derivative, with $\mu = 0.5$ and $r=0$,  affects the form of the solution over time. We let $k = 0.09$, since the previous experiment has shown that this will give a strong effect from the nonlinearity without causing shocks to form, at least  up to time $T=4$.

In these experiments we consider $\Omega = (0,40)$ and  the initial data 
$$
    u_0 = 5e^{-\frac{(x-20)^2}{2}} \quad  \text{ and } \quad  v_0= 0.
$$

Figure~\ref{fig:alpchange} shows the progression of the wave up to time $T=4$ over regular intervals for different values of $a$. For $a=10$  a travelling wave does not form, rather the solution attempts to disperse, and after the initial damping from $t_n=0$ to $t_n = 0.8$ the solution is minimally damped. For $a=0.1$  the nonlinearity has more control, since it appears almost identical to the solution with no involvement of the fractional derivative ($a = 0$). Lastly, the case  $a=1$ shows a balance of effects from the strong damping and nonlinear terms. Finally, for $a = 0$, letting the experiment run until $T = 8$, a shock seems to begin to form; see Figure~\ref{fig:alpchangeallend_longwest}.

\begin{figure}
\includegraphics[width=.8\textwidth]{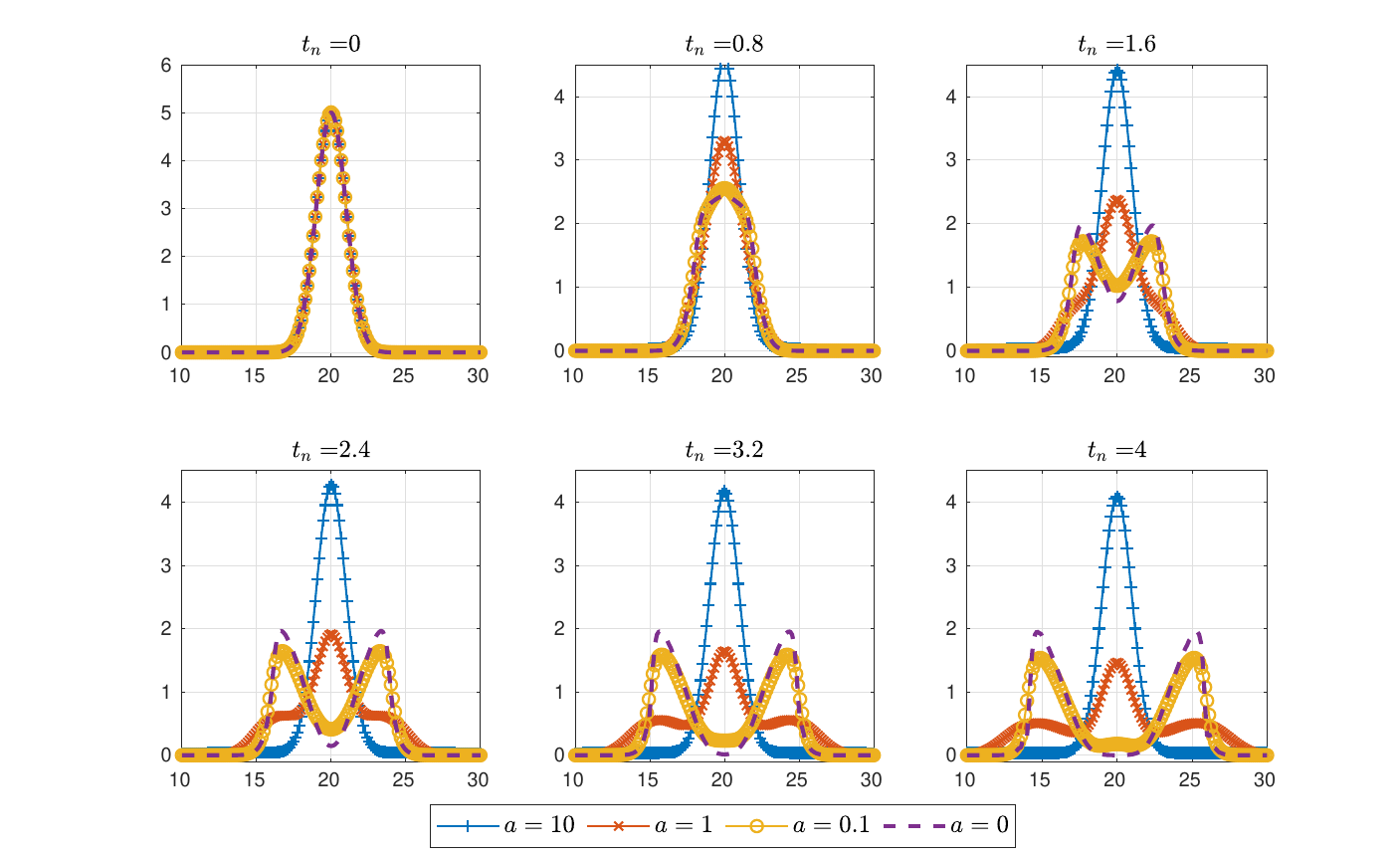}
        \caption{\label{fig:alpchange} Solution of \eqref{eq:1} at various time points up to $T=4$ approximated with the scheme \eqref{equ:full_disc} with $k=0.09$,  $\mu = 0.5$, and $r=0$, for various values of constant $a$.}
\end{figure}



\begin{figure}[H]
    \centering
    \includegraphics[scale=0.5]{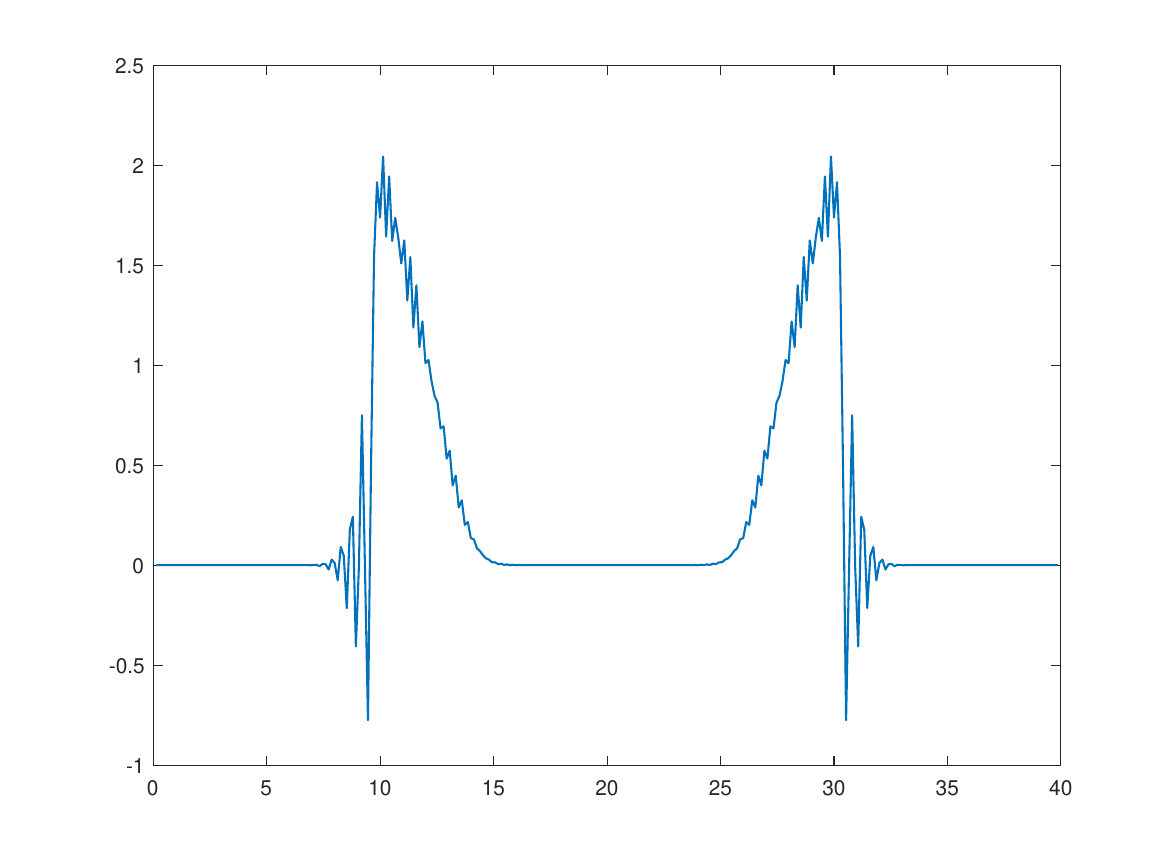}
    \caption{\label{fig:alpchangeallend_longwest} Solution of \eqref{eq:1} at $T=8$ approximated with the scheme \eqref{equ:full_disc} with $k=0.09$ and fixing $a=0$ to remove the fractional derivative.}
\end{figure}

\subsubsection{Test 4: Changing $\mu$}

These experiments demonstrate the effect the order of the fractional operator has on the  solution over time. 
In  Figure~\ref{fig:muchangeallend} we show the solution at $T=4$ with different values of $\mu$, with and without the nonlinearity. 


\begin{figure}[H]
     \centering
     \begin{subfigure}[b]{0.49\textwidth}
         \centering
         \includegraphics[width = \textwidth]{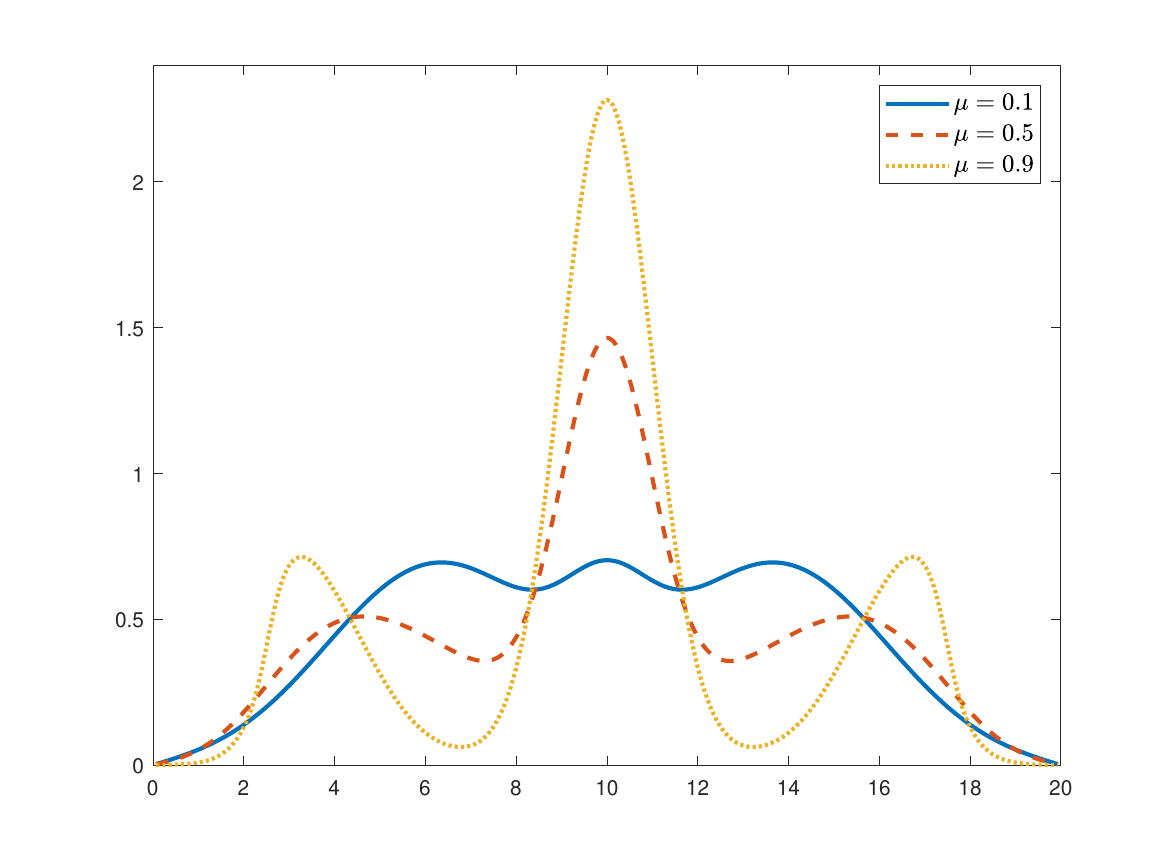}
         \caption{\label{fig:muchangeallendNL} Nonlinear $k = 0.09$}
     \end{subfigure}
     \begin{subfigure}[b]{0.49\textwidth}
         \centering
         \includegraphics[width = \textwidth]{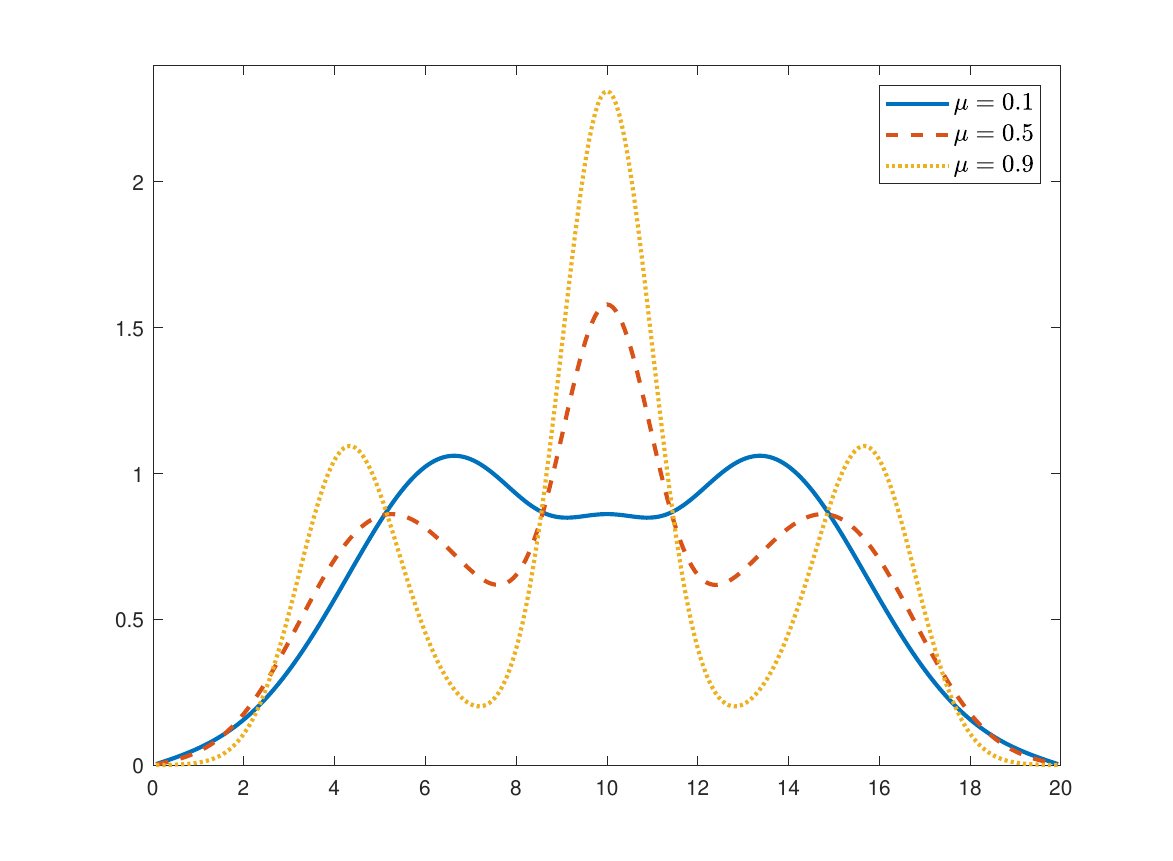}
         \caption{\label{fig:muchangeallendL} Linear $k = 0$}
     \end{subfigure}
        \caption{\label{fig:muchangeallend} Solution of \eqref{eq:1} at the end time $T=4$ approximated with the scheme \eqref{equ:full_disc} with $a=1$, $r=0$ and varying values of $\mu \in (0,1)$.}
\end{figure}

\subsubsection{Test 5: Changing $r$}

In this experiment we vary the value of $r$, recalling that we choose $\beta = \betaA$ and that $\hat{\betaA}(z) = (z+r)^{-\mu}$.

Considering the previously stated initial conditions \eqref{equ:inidata} and  fixing $a = 1$, $\mu = 0.5$, we compute the approximate solution up to final time $T=4$. The solutions for various values of $r$ at the final time T for the nonlinear ($k = 0.09$) and linear ($k = 0$) cases are presented in Figure~\ref{fig:rchangeallend} where we see that as $r$ gets larger the fractional operator displays weaker dispersive behaviour and the solution looks more like a travelling wave solution. 


\begin{figure}[H]
     \centering
     \begin{subfigure}[b]{0.49\textwidth}
         \centering
         \includegraphics[width = \textwidth]{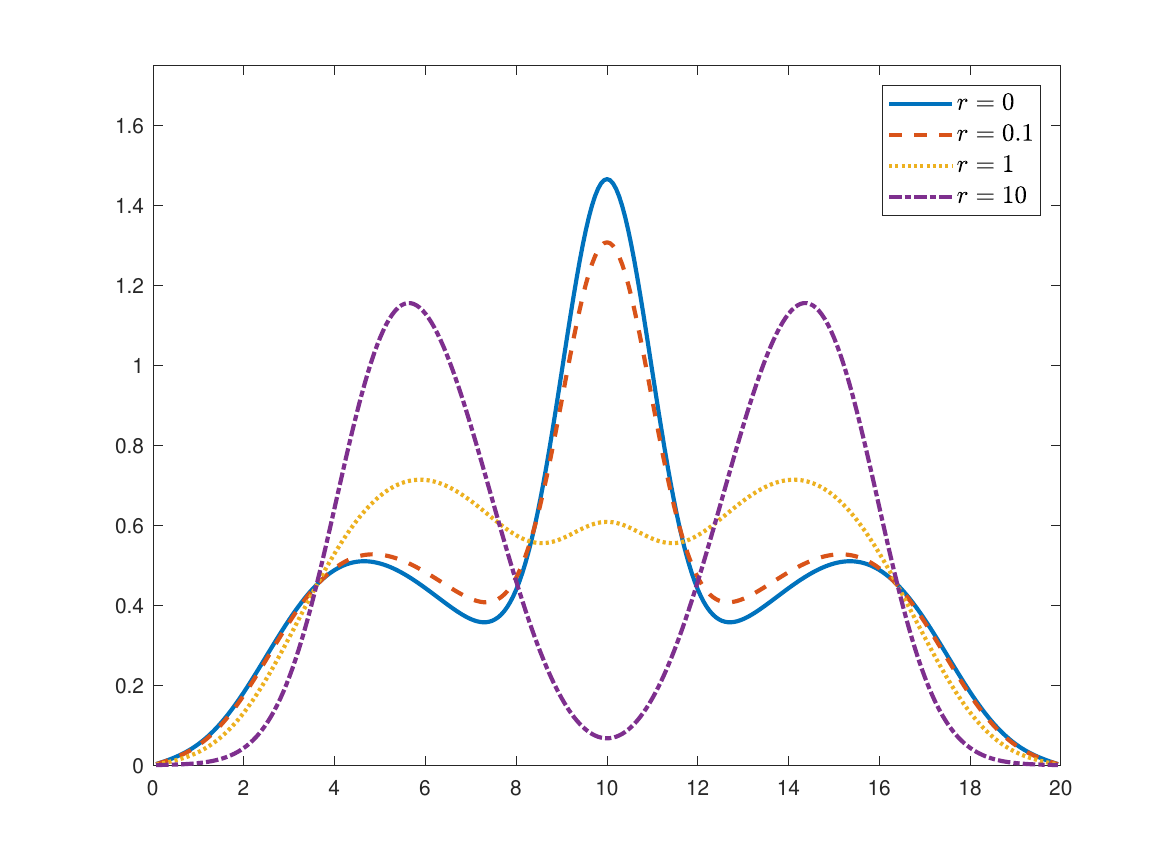}
         \caption{\label{fig:rchangeallendNL} Nonlinear $k = 0.09$}
     \end{subfigure}
     \begin{subfigure}[b]{0.49\textwidth}
         \centering
         \includegraphics[width = \textwidth]{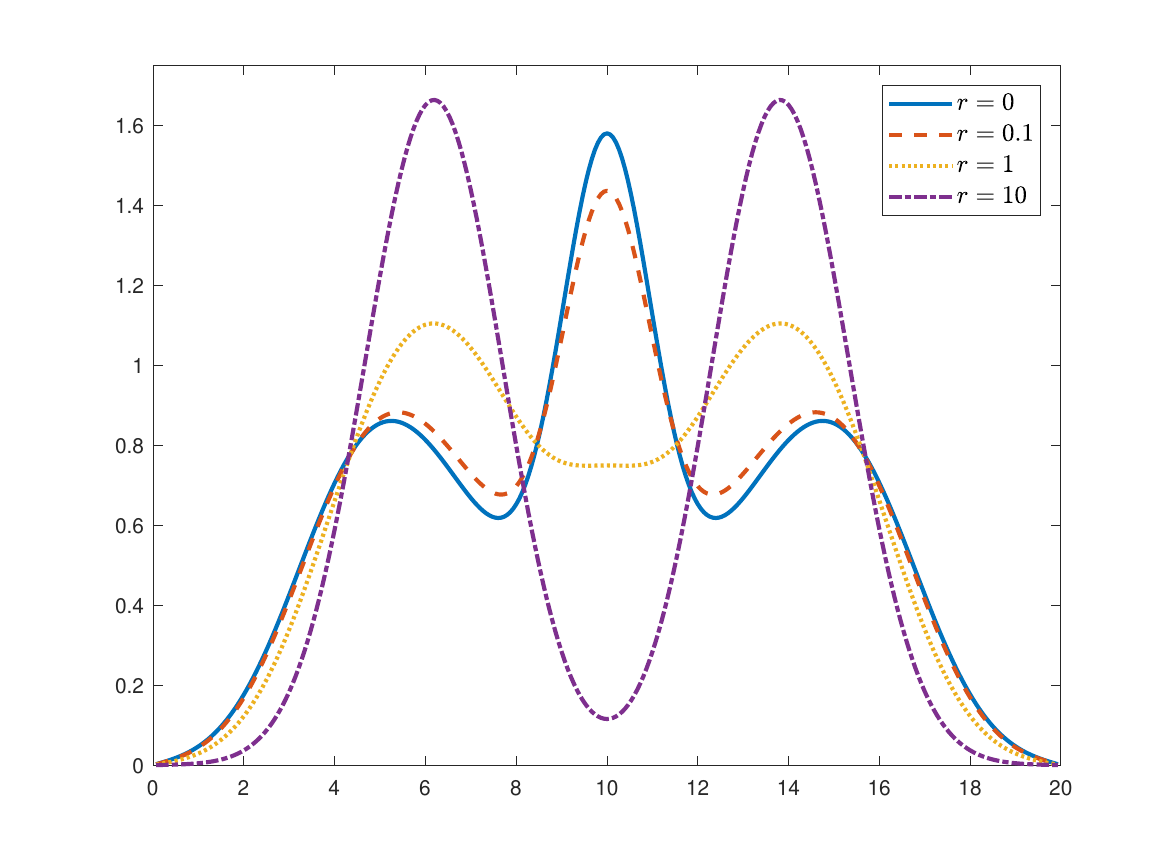}
         \caption{\label{fig:rchangeallendL} Linear $k = 0$}
     \end{subfigure}
        \caption{\label{fig:rchangeallend} Solution of \eqref{eq:1} at the end time $T=4$ approximated with the scheme \eqref{equ:full_disc} with $a=1$, $\mu = 0.5$ and varying values of $r$.}
\end{figure}

\subsection{2D Experiments}
\label{subsec:2Dnum}

In this section we present results of numerical simulations  for \eqref{eq:1} in 2D, where $\Omega = (-1,1)^2$ is a square. We let  the exact solution be
\begin{equation}\label{exact_2D}
    u(x,t) = \left( \sin ( 24 t) + \cos (12t)\right) \sin( \pi x) \sin (\pi y) ,
\end{equation}
and choose a corresponding source term $f$ so as to obtain 
the fully discrete system
\begin{equation}
\label{equ:NLW_with_f}
\begin{split}
    \big\langle \big( 1-2 k \{u^h\}_n \big) D^2 u_n^h ,v  \big\rangle + \big\langle \nabla\{  u^h \}_n,\nabla v \big\rangle + &a \big\langle \beta \ast_{\tstep} D \nabla  u_n^h , \nabla v \big\rangle \\&= 2k \big\langle \big( D u_n^h \big)^2 , v \big\rangle + \big\langle f^h, v \big\rangle,
\end{split}
\end{equation}
where standard CQ without the correction term is used.
We again use $\beta = \betaA$ and the various parameters are set to 
\[
a = 1, \quad k = 0.09, \quad \mu = 0.5, \quad r = 0.
\]
The experiments were performed using the finite element library Netgen/NGSolve package \cite{netgen1}.  In Figure~\ref{fig:2Dsol}, we show the projection of the initial data $u_0$ onto the finite element space, as well as the underlying automatically constructed triangular mesh.

\begin{figure}[H]
    \centering
    \includegraphics[scale = 0.45]{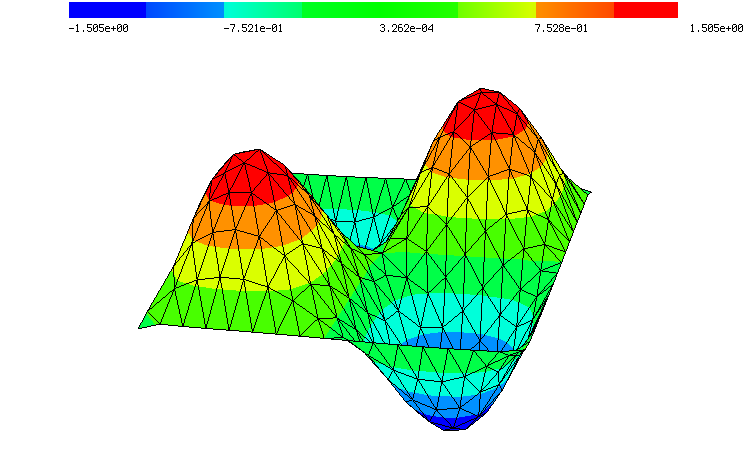}
    \caption{\label{fig:2Dsol} Projection of the initial data on the space of piecewise linear finite elements. The triangulation constructed by  Netgen/NGSolve  \cite{netgen1} is also seen.}
\end{figure}

To examine the convergence rate we compute the maximum $L^2$-error, given by 
\begin{equation}
    \max_{1\leq n\leq N} \left\| u_n - u(t_n) \right\|_{L^2}
\end{equation}
on increasingly finer meshes. In Figure~\ref{fig:2Dconvergence} we see, as expected, $\O(\tstep)$ convergence.

\begin{figure}[H]
    \centering
    \includegraphics[scale = 0.5]{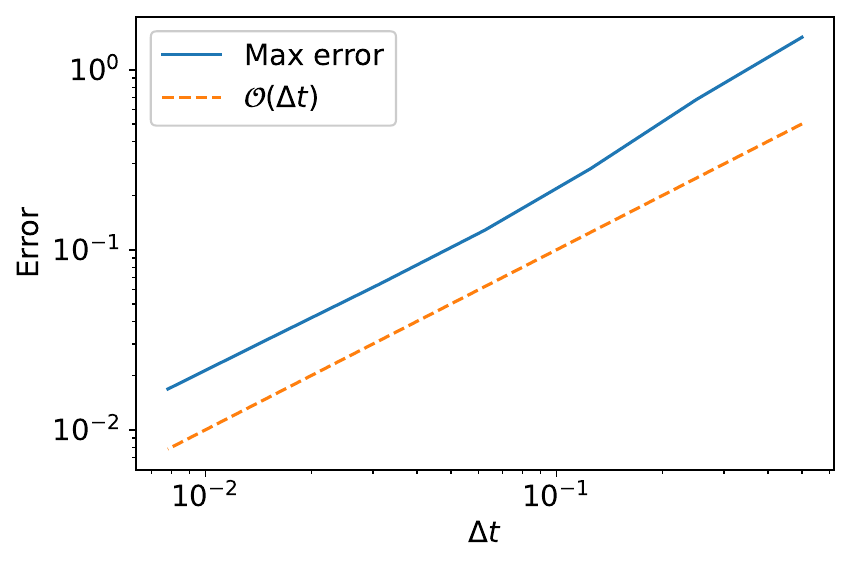}
    \caption{\label{fig:2Dconvergence} Maximum $L^2$-error of the approximated solution to \eqref{eq:1} generated using the scheme~\eqref{equ:NLW_with_f}.}
\end{figure}


\bibliographystyle{amsplain}
\bibliography{notes_biblio_2}

\section*{Appendix: Estimation details for the well-possedness proof}

To derive the first estimate in  Lemma~\ref{lem_apriori}, we 
consider $\partial_t u$ as a test function in  in the weak formulation of~\eqref{main_integro_lin} to obtain 
$$
\begin{aligned} 
 \int_{\Omega_\tau} \Big[\partial_t\big((1- 2k\tilde u) |\partial_t u|^2 \big) + 
 \partial_t |\nabla u|^2  + 2 a 
 \beta \ast \partial_t \nabla u  \partial_t \nabla u  \Big] dx dt  
 \\
 \leq  \int_{\Omega_\tau} 2k |\partial_t \tilde u| 
 |\partial_t  u |^2 dx dt,  
\end{aligned}
$$
for $\tau \in (0,T]$, where $\Omega_\tau = (0,\tau) \times \Omega$. 
Using the nonnegativity of the third term on the left-hand side and the regularity  of $\tilde u$,  and applying the Gr\"onwall inequality we obtain the first estimate in the proof of the lemma.

Considering $-\Delta \partial_t u$ as a test function for \eqref{main_integro_lin} and integrating by parts in the first term and in the term on the right-hand side we obtain
$$
\begin{aligned} 
& \int_{\Omega_\tau} \Big[\partial_t \big((1- 2 k \tilde u)|\partial_t \nabla u |^2\big) + \partial_t |\Delta u |^2 + 2 a \beta\ast \partial_t\Delta u \partial_t \Delta u \Big]dx dt    \\
& = \int_{\Omega_\tau} \Big[ 4k \nabla \tilde u \partial_t^2 u \partial_t \nabla u   + 2k \partial_t \tilde u |\partial_t \nabla u|^2
+ 4k \partial_t \nabla \tilde u \partial_t u \partial_t \nabla u \Big]dx dt, \; \text{ for } \tau \in (0, T]. 
\end{aligned} 
$$
Using that 
\begin{equation}\label{second_deriv_strong}
\partial_t^2 u = \frac 1{ 1- 2k \tilde u} \Big(\Delta u + a \beta\ast \partial_t \Delta u +  2k\partial_t \tilde u \partial_t u \Big), 
\end{equation}
we can estimate the terms on the right-hand side as 
$$
\begin{aligned} 
  \int_{\Omega_\tau} |\nabla \tilde u \partial_t^2 u \partial_t \nabla u| dx dt 
 \leq \frac{1}{2\kappa}\int_{\Omega_\tau} |\nabla \tilde u|\Big(|\Delta u|^2 +  |\partial_t \nabla u|^2 + 4k |\partial_t \tilde u| |\partial_t u||\partial_t \nabla u| \Big) dx dt 
 \\
  \qquad + \frac{1}{2\kappa^2}  \frac 1 \varsigma\int_{\Omega_\tau} |\nabla \tilde u|^2 |\partial_t \nabla u|^2 dx dt  +  \frac{\varsigma a^2} 2 \int_{\Omega_\tau} |\beta\ast \partial_t \Delta u|^2dx dt\\
 \quad  \leq \frac{1}{2\kappa} \int_0^\tau \|\nabla \tilde u\|_{L^\infty}\Big(\|\Delta u\|^2_{L^2}   
 + \Big[1+ \frac 1{\kappa\varsigma} \|\nabla \tilde u\|_{L^\infty} + 4k C_\Omega \|\nabla \partial_t \tilde u\|^2_{L^2}\Big] \|\partial_t \nabla u\|^2_{L^2}\Big) dt 
 \\
 \qquad  + \frac{\varsigma a^2}2 \int_{\Omega_\tau}  |\beta\ast \partial_t \Delta u|^2 dx dt 
\end{aligned}
$$
and 
$$
\begin{aligned} 
 \int_{\Omega_\tau} \Big[2 k \partial_t \tilde u |\partial_t \nabla u|^2
+ 4k \partial_t \nabla \tilde u \partial_t u \partial_t \nabla u \Big]dx dt
\leq 2k\int_0^T \Big[  \|\partial_t \tilde u\|_{L^\infty} \|\partial_t \nabla u\|^2_{L^2} \quad \\
+  2 C_\Omega  \|\nabla \partial_t \tilde u\|_{L^4} \|\nabla \partial_t u\|_{L^2}^2 \Big]  dt.
\end{aligned}
$$
Then using estimate in Lemma~\ref{lem:time_lower} and applying the Gr\"onwall inequalities yields the second estimate in the proof of Lemma~\ref{lem_apriori}.

Applying $\Delta$ to  \eqref{main_integro_lin} and taking $\Delta \partial_t u$ as a test function in the weak formulation of the problem  imply 
\begin{equation}\label{eq_Laplace_1}
\begin{aligned} 
& \int_{\Omega_\tau} \Big(\partial_t \big((1- 2 k \tilde u)|\partial_t \Delta u |^2\big) + \partial_t |\Delta \nabla u |^2 + 2 a \, \beta\ast \partial_t\Delta \nabla u \partial_t \Delta \nabla u \Big)dx dt    \\
& \quad = 
\int_{\Omega_\tau} \Big[
4k \Delta \tilde u \partial_t^2 u \partial_t \Delta u  
+ 8 k \nabla \tilde u \partial_t^2 \nabla u \partial_t \Delta u
+ 2k \partial_t \tilde u |\partial_t \Delta u|^2 
\\
& \qquad \qquad + 4k \partial_t \Delta \tilde u \partial_t u \partial_t \Delta u  + 
8k  \partial_t \nabla \tilde u \partial_t \nabla u \partial_t \Delta u \Big]dx dt = I_1 + I_2 + I_3. 
\end{aligned} 
\end{equation}
Using \eqref{second_deriv_strong} together with 
\begin{equation} \label{second_deriv_strong_1}
\begin{aligned} 
\partial_t^2 \nabla u  = & \frac 1{ 1- 2k \tilde u} \Big(\Delta \nabla u + a \beta\ast \partial_t \Delta \nabla u +  2k\partial_t \nabla\tilde u \partial_t u +  2k\partial_t \tilde u \partial_t\nabla u \Big)\\
& + \frac{2k\nabla \tilde u }{ (1- 2k \tilde u)^2} \Big(\Delta u + a \beta\ast \partial_t \Delta u +  2k\partial_t \tilde u \partial_t u \Big),
\end{aligned} 
\end{equation} 
the terms on the right-hand side in \eqref{eq_Laplace_1} can be estimated as 
$$
\begin{aligned} 
& |I_1| \leq  \frac {2k}\kappa \int_{\Omega_\tau} \Big( 2|\Delta \tilde u||\Delta u||\partial_t \Delta  u|  + 4k |\Delta \tilde u| |\partial_t \tilde u||\partial_t u| |\partial_t \Delta  u| \Big) dx dt 
 \\
 &\qquad  + 2k^2/(\kappa^2 \varsigma) \int_0^\tau \|\Delta \tilde u \|^2_{L^4} 
 \|\partial_t \Delta  u\|^2_{L^2} dt + 2\varsigma a^2 \int_0^\tau \|\beta\ast \partial_t \Delta u\|^2_{L^4} dt\\
&  \leq \frac {2k}\kappa \int_0^\tau \Big[ \|\Delta  \tilde u\|_{L^4} \big(C_\Omega\|\Delta \nabla u \|^2_{L^2} +  \|\partial_t \Delta  u\|^2_{L^2} \big)
 + 2k \Big(4 C_\Omega \|\Delta \tilde u\|_{L^4}\|\partial_t \nabla \tilde u\|_{L^2}\\
& \qquad  + \frac 1{\kappa \varsigma} \|\Delta \tilde u \|_{L^4}^2 \Big)\|\partial_t \Delta u \|_{L^2}^2 \Big] dt 
 + 
 2\varsigma a^2 C_\Omega \int_0^\tau \|\beta\ast \partial_t \Delta \nabla u\|^2_{L^2} dt
\end{aligned}
$$
and
$$
\begin{aligned}
& |I_2| 
 \leq \frac {4k}{\kappa} \hspace{-0.1 cm} \int_{\Omega_\tau} \hspace{-0.2 cm} |\nabla \tilde u| \Big[ |\Delta \nabla u|^2 + 4k \big( |\nabla \partial_t \tilde u| |\partial_t u| + 
  |\partial_t \tilde u| |\nabla \partial_t u| + |\partial_t \Delta  u|\big) |\partial_t \Delta  u|  \Big] dx dt 
 \\
 & \quad + \frac{8k^2}{\kappa^2 \varsigma}\int_0^\tau \|\nabla \tilde u\|^2_{L^\infty} \|\partial_t \Delta  u\|^2_{L^2} dt 
  + 
 \varsigma a^2  \int_{\Omega_\tau} 
 |\beta\ast \partial_t \Delta \nabla u|^2 dx dt\\
 & \quad + \frac{16k^2}{\kappa}  \int_{\Omega_\tau}  \Big(|\nabla \tilde u|^2 |\Delta u|
 |\partial_t \Delta u| + 2k |\nabla \tilde u|^2 |\partial_t \tilde u||\partial_t u||\partial_t \Delta  u| \Big) dx dt\\
 & \quad + \frac{16k^2}{\kappa}a  \int_{\Omega_\tau} |\nabla \tilde u|^2
 |\beta\ast \partial_t \Delta u| |\partial_t \Delta u| dx dt 
 \leq 8 \varsigma a^2 \int_0^\tau \|\beta \ast \partial_t \Delta \nabla u \|^2_{L^2} dt
 \\
&\quad   + 
 \frac {4k}{\kappa} \int_0^\tau  \|\nabla\tilde u\|_{L^\infty} 
 \Big[\big(1+ 2k C_\Omega \|\nabla \tilde u\|_{L^4}\big)\|\Delta \nabla u \|_{L^2}^2  + \Big(1 + 8k \|\nabla \partial_t \tilde u\|_{L^2} \\
 &\quad  + \frac{2k}{\kappa \varsigma} \|\nabla \tilde u\|_{L^\infty}+ 2k C_\Omega \|\nabla \tilde u\|_{L^4}(1 + 8 k \|\partial_t \nabla \tilde u\|_{L^2})   + \frac{2k^3}{\kappa \varsigma} \|\nabla \tilde u\|^3_{L^4} \Big)  \|\partial_t \Delta u \|^2_{L^2} \Big]  dt  .
\end{aligned}
$$
The last three terms on the right-hand side in \eqref{eq_Laplace_1} are estimated as 
$$
\begin{aligned}
|I_3| & \leq 2 k\int_0^\tau \Big[\|\partial_t \tilde u \|_{L^\infty} \|\partial_t \Delta u\|^2_{L^2} + 2 \|\partial_t \Delta \tilde u\|_{L^2}
 \|\partial_t u\|_{L^\infty} \|\partial_t \Delta u \|_{L^2}\\
 & + 4
 \|\partial_t \nabla \tilde u \|_{L^4} \|\partial_t \nabla u \|_{L^4} \|\partial_t \Delta u \|_{L^2} \Big] dt
 \leq 2 k \, C_\Omega\int_0^\tau   \|\partial_t \Delta \tilde u\|_{L^2}  \|\partial_t \Delta u \|^2_{L^2} dt.
\end{aligned}
$$
Collecting the estimates from above and applying the Gr\"onwall inequality yields the third estimate in the proof of Lemma~\ref{lem_apriori}.

To show that $\mathcal T \colon \mathcal K \to \mathcal K$ is a contraction in the proof of Theorem~\ref{thm:main_cont} we consider \eqref{main_integro_lin} for $\tilde u_1$ and $\tilde u_2$ in $\mathcal K$
and, taking $-\Delta \partial_t (u_1 - u_2)$ as a test function for the difference of the corresponding equations, obtain
$$
\begin{aligned} 
 & \int_{\Omega_T} \Big[\partial_t\big[(1- 2k \tilde u_1) |\nabla \partial_t (u_1 - u_2)|^2 \big] + \partial_t |\Delta (u_1-u_2)|^2 
\\
& \qquad +2 a \beta \ast \Delta \partial_t (u_1- u_2)\Delta \partial_t (u_1- u_2) \Big]dx dt 
 \\
&  = 4k \int_{\Omega_T} \Big[\nabla(\tilde u_1- \tilde u_2) \partial_t^2 u_2 \nabla \partial_t (u_1 - u_2) +  (\tilde u_1- \tilde u_2) \partial_t^2 \nabla u_2 \nabla \partial_t (u_1 - u_2) \\
& \qquad  +  \nabla \tilde u_1 \partial_t^2(u_1-u_2) \nabla \partial_t (u_1- u_2) \Big]dx dt 
 \\
&  + 2k  \int_{\Omega_T} \Big[ \partial_t \tilde u_1 |\nabla \partial_t(u_1- u_2)|^2 + 
 2\partial_t \nabla \tilde u_1  \partial_t(u_1- u_2)\nabla \partial_t(u_1- u_2) \\
& \qquad  + 2 \partial_t (\tilde u_1- \tilde u_2) \nabla \partial_t u_2 \nabla \partial_t (u_1 - u_2) +  2 \partial_t \nabla (\tilde u_1- \tilde u_2) \partial_t u_2 \nabla \partial_t (u_1 - u_2) \Big]dx dt,
\end{aligned}
$$
where $\Omega_T = (0,T) \times \Omega$.
Using \eqref{second_deriv_strong} and \eqref{second_deriv_strong_1}, the terms on the righ-hand side of the last equality are estimated  as
$$
\begin{aligned} 
|I_1|\leq  & \int_{\Omega_T} |\nabla(\tilde u_1- \tilde u_2) \partial_t^2 u_2 \nabla \partial_t (u_1 - u_2)| dx dt \leq 
 \frac 1 { \kappa} \int_0^T \|\nabla (\tilde u_1-\tilde u_2)\|_{L^4} \Big(\|\Delta u_2\|_{L^4} \\
 & + a \|\beta \ast \partial_t \Delta u_2\|_{L^4} + 2k \|\partial_t u_2\|_{L^8} \|\partial_t \tilde u_2\|_{L^8} \Big) \|\nabla \partial_t (u_1 - u_2) \|_{L^2} dt,  
\end{aligned}
$$
$$
\begin{aligned} 
& |I_2| \leq   \int_{\Omega_T} |(\tilde u_1- \tilde u_2) \partial_t^2 \nabla u_2 \nabla \partial_t (u_1 - u_2)| dx dt \leq \frac 1 {\kappa} 
 \int_0^T \hspace{-0.2 cm} \|\tilde u_{1,2}\|_{L^\infty}
 \Big( a\|\beta\ast \Delta \nabla\partial_t  u_2\|_{L^2} 
 \\
 & + \|\Delta \nabla u_2\|_{L^2}+
 2k \|\nabla \tilde u_2\|_{L^\infty}\Big[ a\|\beta\ast \Delta  \partial_t u_2\|_{L^2} + 2k \|\partial_t u_2 \|_{L^4} \|\partial_t \tilde u_2 \|_{L^4} + \|\Delta u_2\|_{L^2} \Big]
\\
& \quad + 2k \|\partial_t \nabla \tilde u_2\|_{L^4} \|\partial_t u_2 \|_{L^4}+
2k \|\partial_t \nabla  u_2\|_{L^4} \|\partial_t \tilde u_2 \|_{L^4}
 \Big) \|\nabla \partial_t (u_1 - u_2)\|_{L^2} dt, 
\end{aligned}
$$
where $\tilde u_{1,2} :=\tilde u_1-\tilde u_2$,
$$
\begin{aligned} 
|I_3| & \leq  \int_{\Omega_T} |\nabla \tilde u_1 \partial_t^2(u_1-u_2) \nabla \partial_t (u_1- u_2) |dx dt \\
& \leq 
\frac 1 { \kappa } \int_0^T \|\nabla \tilde u_1\|_{L^\infty} 
\Big(\Big[2k \|\partial_t  u_2\|_{L^4} \|\nabla \partial_t(\tilde u_1- \tilde u_2)\|_{L^2}  + a \|\beta\ast \partial_t \Delta (u_1- u_2)\|_{L^2}\\
&\quad  + \|\Delta (u_1- u_2)\|_{L^2}\Big] \|\nabla \partial_t (u_1 - u_2)\|_{L^2} 
+ 2k \|\partial_t \tilde u_1\|_{L^4} \|\nabla \partial_t(u_1- u_2)\|^2_{L^2}  \Big) dt, 
\end{aligned} 
$$
and 
$$
\begin{aligned} 
|I_4|
& \leq 2k \int_0^T \Big[\Big( \|\partial_t \tilde u_1\|_{L^\infty}  + 2 \|\nabla \partial_t \tilde u_1 \|_{L^4} \Big) \|\nabla \partial_t(u_1- u_2)\|^2_{L^2} 
\\
& \quad + 2 \Big(\|\nabla \partial_t u_2\|_{L^4} + \|\partial_t u_2\|_{L^\infty}\Big) \|\nabla \partial_t(\tilde u_1- \tilde u_2)\|_{L^2} \|\nabla \partial_t(u_1- u_2)\|_{L^2} \Big]dt. 
\end{aligned} 
$$
Then using that  $\tilde u_j\in \mathcal K$, for $j=1,2$, and estimates for $u_j$ in Lemma~\ref{lem_apriori} we obtain the contraction inequality in Theorem~~\ref{thm:main_cont}, which, by applying the Banach fixed point theorem,   ensures the existence of a unique fixed point  of the map $\mathcal T$ and hence the existence of a unique solution of the nonlinear problem~\eqref{main_integro}.

\end{document}